\documentclass[11pt]{amsart}

\usepackage{amsfonts,latexsym}
\usepackage{amsmath}
\usepackage{amstext}
\usepackage{amssymb}\usepackage{mathrsfs}
\usepackage{color}
\usepackage{graphicx}
\usepackage{enumerate}
\graphicspath{{figures/}}

\newcommand{\R}{\mathbb{R}}

\newcommand{\N}{{\mathbb{N}}}

\newtheorem{theorem}{Theorem}
\newtheorem{lemma}{Lemma}
\newtheorem{proposition}{Proposition}
\newtheorem{corollary}{Corollary}
\newtheorem{definition}{Definition}

\newtheorem{assumption}{Assumption}
\newtheorem{remark}{Remark}
\newtheorem{problem}{Problem}

\numberwithin{equation}{section}

\begin{document}

\title[Opinion Dynamics with Decaying Confidence]{Opinion Dynamics with Decaying Confidence:\\ Application to Community Detection in Graphs}
\thanks{This work was supported by UJF-MSTIC (CARESSE Project).}


\author[Irinel-Constantin Mor\u{a}rescu]{Irinel-Constantin Mor\u{a}rescu}
\address{Laboratoire Jean Kuntzmann \\
Universit\'e de Grenoble \\
B.P. 53, 38041 Grenoble, France} \email{irinel-constantin.morarescu@inrialpes.fr}

\author[Antoine Girard]{Antoine Girard}
\address{Laboratoire Jean Kuntzmann \\
Universit\'e de Grenoble \\
B.P. 53, 38041 Grenoble, France} \email{antoine.girard@imag.fr}


\maketitle

\vspace{-0.5cm}
\begin{abstract}
We study a class of discrete-time multi-agent systems modelling
opinion dynamics with decaying confidence. We consider a network
of agents where each agent has an opinion. At each time step, the
agents exchange their opinion with their neighbors and update it
by taking into account only the opinions that
differ from their own less than some confidence bound. This
confidence bound is decaying: an agent gives repetitively
confidence only to its neighbors that approach sufficiently fast
its opinion. Essentially, the agents try to reach an agreement
with the constraint that it has to be approached no slower than a
prescribed convergence rate. Under that constraint, global
consensus may not be achieved and only local agreements may be
reached. The agents reaching a local agreement form communities
inside the network. In this paper, we analyze this opinion
dynamics model: we show that communities correspond to
asymptotically connected component of the network and give an
algebraic characterization of communities in terms of eigenvalues
of the matrix defining the collective dynamics. Finally, we apply
our opinion dynamics model to address the problem of community
detection in graphs. We propose a new formulation of the community
detection problem based on eigenvalues of normalized Laplacian
matrix of graphs and show that this problem can be solved using
our opinion dynamics model. 
We consider three examples of networks, and compare 
the communities we detect with those obtained by existing algorithms based on modularity optimization.
We show that our opinion dynamics model not only provides an appealing approach
to community detection but that it is also effective.
\end{abstract}

\section{Introduction}
The analysis of multi-agent systems received an increasing interest in the past decades.
In such systems, a set of agents interact according to simple local rules in order to achieve some global coordinated behavior.
The most widely studied problem is certainly the consensus or agreement problem where each agent in the network
maintains a value and repetitively averages its value with those of its neighbors, resulting in all the agents in the network
reaching asymptotically a common value. It is to be noted that the graph of interaction describing the network of agents is generally not fixed and may vary in time. Conditions ensuring consensus have been established by various authors including~\cite{jadbabaie2003,blondel2005,moreau2005,ren2005} (see~\cite{olfati2007} for a survey).
More recently, there have been several works providing estimations of the rate of convergence towards the consensus value~\cite{olshevsky2006,angeli2007,zhou2009}.

In this paper, we adopt a different point of view. We consider a
discrete-time multi-agent system where the agents try to reach an
agreement with the constraint that the consensus value must be
approached no slower than a prescribed convergence rate. Under
that constraint, global consensus may not be achieved and the
agents may only reach local agreement. We call communities the
subsets of agents reaching a consensus. Our model can be
interpreted in terms of opinion dynamics. Each agent has an
opinion. At each time step, the agent receives the opinions of its
neighbors and then updates its opinion by taking a weighted
average of its opinion and the opinions of its neighbors that are
within some confidence range of its own. The confidence
ranges are getting smaller at each time step: an agent gives
repetitively confidence only to the neighbors that approach
sufficiently fast its own opinion. This can be seen as a model for a
negotiation process where an agent expects that its neighbors move
significantly towards its opinion at each negotiation round in
order to keep negotiating. Our model can be seen as an extension
of the opinion dynamics with bounded confidence proposed by Krause
in~\cite{krause1997} and studied in~\cite{krause2002,blondel2007}.

We analyze our opinion dynamics model by first studying the
relation between asymptotic agreement of a subset of agents and
the fact that they are asymptotically connected. We show that
under suitable assumptions, these are actually equivalent (i.e.
communities correspond to asymptotically connected component of
the network) except for a set of initial opinions of measure $0$.
We then give an algebraic characterization of communities in terms
of eigenvalues of the matrix defining the collective dynamics.

Finally, we apply our opinion dynamics model to address the
problem of community detection in graphs. In the usual sense,
communities in a graph are groups of vertices such that the
concentration of edges inside communities is high with respect to
the concentration of edges between communities. Given the
increasing need of analysis tools for understanding complex
networks  in social sciences, biology, engineering or economics,
the community detection problem has attracted a lot of attention
in the recent years (see the extensive
survey~\cite{fortunato2009}). The problem of community detection
is however not rigorously defined mathematically.
Some formalizations of this problem have been proposed in
terms of optimization of quality functions such as modularity~\cite{newman2004}
or partition stability~\cite{lambiotte2009}.
We propose a new formulation of this
problem based on eigenvalues of normalized Laplacian matrix of
graphs  and show that this problem  can be solved using
our opinion dynamics model. 
We consider three examples of networks, and compare 
the communities that we detect with those obtained by the modularity optimization algorithms presented in~\cite{newman2006,blondel2008}.
We show that our opinion dynamics model not only provides an appealing approach
to community detection but that it is also effective.

\section{Opinion Dynamics with Decaying Confidence}

\subsection{Model Description}
We study a discrete-time multi-agent model. We consider a set of
$n$ {\it agents}, $V=\{1,\dots,n\}$. A relation $E\subseteq V
\times V$ models the {\it interactions} between the agents. We
assume that the relation is symmetric ($(i,j) \in E$ iff $(j,i)
\in E$) and anti-reflexive ($(i,i)\notin E$). $V$ is the set of
vertices and $E$ is the set of edges of an undirected graph
$G=(V,E)$, describing the network of agents. Each agent $i\in V$
has an {\it opinion} modelled by a real number $x_i(t)\in \R$.
Initially, agent $i$ has an opinion $x_i(0)=x_i^0$ independent
from the opinions of the other agents. Then, at every time step,
the agents update their opinion by taking a weighted average of
its opinion and opinions of other agents:
\begin{equation}
\label{eq:dynamics}
x_i(t+1) = \sum_{j=1}^{n} p_{ij}(t) x_j(t)
\end{equation}
with the coefficients $p_{ij}(t)$ satisfying
\begin{equation}
\label{eq:neigh}
\forall i, j\in V,\;
\left(p_{ij}(t) \ne 0 \iff j \in \{i\}\cup N_i(t) \right)
\end{equation}
where $N_i(t)$ denotes the {\it confidence neighborhood} of agent $i$ at time $t$:
\begin{equation}
\label{eq:neigh2}
N_i(t) = \left\{j\in V |\; \left((i,j) \in E\right) \land \left(|x_i(t)-x_j(t)| \le R \rho^t\right) \right\}
\end{equation}
with $R>0$ and $\rho\in (0,1)$ model parameters.

\begin{remark}\label{dependence}
It is noteworthy that the confidence neighborhoods $N_i(t)$ and the coefficients $p_{ij}(t)$ actually depend also
on the opinions $x_1(t),\dots,x_n(t)$.
For
the sake of simplicity and in order to reduce the length of the
equations we keep the notations $p_{ij}(t)$ and $N_i(t)$ pointing
out just the variation in time of these quantities.
\end{remark}

We make the following additional assumptions:
\begin{assumption}[Stochasticity]\label{assum1} For $t\in \N$, the coefficients $p_{ij}(t)$ satisfy
\begin{enumerate}[(a)]
\item $p_{ij}(t) \in [0,1]$, for all $i,j\in V$.

\item $\sum_{j=1}^{n} p_{ij}(t) =1$, for all $i\in V$.
\end{enumerate}
\end{assumption}

This model can be interpreted in terms of opinion dynamics. At
each time step $t$, agent $i\in V$ receives the opinions of its
neighbors in the graph $G$. If the opinion of $i$ differs from the
opinion of its neighbor $j$ more than a certain threshold $R
\rho^t$, then $i$ does not give confidence to $j$ and does not
take into account the opinion of $j$ when updating its own
opinion. The parameter $\rho$ characterizes the confidence decay
of the agents. Agent $i$ gives repetitively confidence only to
neighbors whose opinion converges sufficiently fast to its own
opinion. This model can be interpreted in terms of negotiations
where agent $i$ requires that, at each negotiation round, the
opinion of agent $j$ moves significantly towards its opinion
in order to keep negotiating with $j$. 

This model is somehow
related to the one discussed in~\cite{chatterjee1977,cohen1986} where agents harden their position by increasing over time the weight assigned to their own opinion.
In our model, the agents implicitely increase also
the weights assigned to their neighbors whose opinion converges
sufficiently fast to their own opinion, by disregarding the 
opinions of the other agents. 
As noticed in~\cite{cohen1986}, hardening the agents positions may hamper the agents to reach an asymptotic
consensus. This will be observed in our model as well.
However, the aim in this paper is not to exogenously
increase the self-confidence of the agents, but to meet a prescribed
convergence speed towards the final opinion profile. 

\begin{remark}
We assume in this paper that $\rho\in (0,1)$. However, let us remark that 
for $\rho=1$ (there is no confidence decay), with a
complete graph $G$ (every agent talks with all the other agents),
 and with coefficients $p_{ij}(t)$ given 
for all $j\in \{i\}\cup N_i(t)$ by
$$
p_{ij}(t)=\frac{1}{1+d_i(t)} \text{ with } d_i(t)=\sum_{j\in N_i(t)}1
$$ 
our model would coincide with
Krause model of opinion dynamics with bounded
confidence~\cite{krause1997,krause2002,blondel2007}.
\end{remark}

Our first result states that the opinion of each agent converges
to some limit value:
\begin{proposition}\label{prop1}
\label{pro:conv}Under Assumption~\ref{assum1} (Stochasticity),
for all $i\in V$, the sequence $(x_i(t))_{t\in \N}$ is convergent.
We denote $x_i^*$ its limit. Furthermore, we have for all $t\in \N$,
\begin{equation}
\label{eq:conv}
|x_i(t)-x_i^*| \le \frac{R}{1-\rho} \rho^t.
\end{equation}
\end{proposition}
\begin{proof} Let $i\in V$, $t\in \N$, we have from (\ref{eq:dynamics}), Assumption~\ref{assum1}  and (\ref{eq:neigh})
\begin{eqnarray*}
|x_i(t+1)-x_i(t)|&=&  \left|\left(\sum_{j=1}^{n} p_{ij}(t) x_j(t) \right) -x_i(t) \right| \\
&=&  \left|\sum_{j=1}^{n} p_{ij}(t)( x_j(t)-x_i(t)) \right|\\
&=&  \left|\sum_{j\in N_i(t)} p_{ij}(t)( x_j(t)-x_i(t)) \right|\\
&\le& \sum_{j\in N_i(t)} p_{ij}(t) |x_j(t)-x_i(t)|
\end{eqnarray*}
Then, it follows  from equation (\ref{eq:neigh2}) that
$$
|x_i(t+1)-x_i(t)| \le \sum_{j\in N_i(t)} p_{ij}(t) R \rho^t
$$
Finally, Assumption~\ref{assum1} gives for all $t\in \N$
$$
|x_i(t+1)-x_i(t)| \le (1-p_{ii}(t)) R \rho^t \le  R \rho^t.
$$
Let $t \in \N$, $\tau \in \N$, then
$$
|x_i(t+\tau)-x_i(t)| \le \sum_{k=0}^{\tau-1} |x_i(t+k+1)-x_i(t+k)| \le \sum_{k=0}^{\tau-1}   R \rho^{t+k}
$$
Therefore,
\begin{equation}
\label{eq:cauchy}
|x_i(t+\tau)-x_i(t)| \le \frac{R}{1-\rho}   \rho^t (1- \rho^\tau)\le \frac{R}{1-\rho} \rho^t
\end{equation}
which shows, since $\rho \in (0,1)$, that the sequence
$(x_i(t))_{t\in \N}$ is a Cauchy sequence in $\R$. Therefore, it
is convergent. Equation~(\ref{eq:conv}) is obtained from
(\ref{eq:cauchy}) by letting $\tau$ go to $+\infty$.
\end{proof}

\begin{remark}
The convergence of each opinion sequence $(x_i(t))_{t\in \N}$ could have been proved using a result from~\cite{lorenz2005}, even for $\rho=1$,
with the additional assumption that the non-zero coefficients $p_{ij}(t)$ are uniformly bounded below by some strictly positive real number. However, the result in~\cite{lorenz2005} does not provide an estimation of the convergence rate which is essential in our subsequent discussions.
\end{remark}

The previous proposition allows us to complete the interpretation
of our opinion dynamics model. The agents try to reach an
agreement with the constraint that the consensus value must be
approached no slower than $O(\rho^t)$. Under that constraint,
global agreement may not be attainable and the agents may only
reach local agreements. We refer to the sets of agents that
asymptotically agree as communities.

\begin{definition} Let $i,j \in V$, we say that agents $i$ and $j$ {\it asymptotically agree}, denoted $i\sim^* j$, if and only if $x_i^*=x_j^*$.
\end{definition}

It is straightforward to verify that $\sim^*$ is an equivalence relation over $V$.
\begin{definition} A {\it community} $C \subseteq V$ is an element of the quotient set $\mathscr{C}=V/\sim^*$.
\end{definition}

Let us remark that the community structure is dependent on the
initial distribution of opinions. In the following, we shall
provide some insight on the structure of these communities. But
first, we need to introduce some additional notations.

\subsection{Notations and Preliminaries}
We define the set of {\it interactions at time $t$}, $E(t) \subseteq V\times V$ as
$$
E(t) =\left \{ (i,j) \in E |\;  |x_i(t)-x_j(t)| \le R \rho^t \right\}.
$$
Let us remark that $(i,j) \in E(t)$ if and only if $j\in N_i(t)$. The {\it interaction graph at time $t$} is then $G(t)=(V,E(t))$. Let us remark that Remark~\ref{dependence} applies also to $E(t)$ and $G(t)$. 

For a set of  agents $I \subseteq V$, the subset of edges of $G$
connecting the agents in $I$ is $E_{I}=E\cap(I \times I)$. Let
$E'\subseteq E_I$ be a symmetric relation over $I$, then the graph
$G'=(I,E')$ is called a {\it subgraph} of $G$. If $I=V$, then the
graph $G'=(V,E')$ is called a {\it spanning subgraph} of $G$. The
set of spanning subgraphs of $G$ is denoted $\mathcal S(G)$. For
all $t\in \N$, $G(t)\in \mathcal S(G)$. Let us remark that the set
$\mathcal S(G)$ is finite: it has $2^{|E|/2}$ elements (because we only consider symmetric relations) where $|E|$
denotes the number of elements in $E$. Given a partition of the
agents $\mathcal P=\{I_1,\dots,I_p\}$, we define the set of edges
$E_\mathcal P=\bigcup_{I\in\mathcal P} E_I$ and the spanning
subgraph of $G$, $G_{\mathcal P}=(V,E_{\mathcal P})$. Essentially,
$G_{\mathcal P}$ is the spanning subgraph of $G$ obtained by
removing all the edges between agents belonging to different
elements of the partition $\mathcal P$. An interesting such graph
is the {\it graph of communities}
$G_{\mathscr{C}}=(V,E_{\mathscr{C}})$ where:
$$
E_{\mathscr{C}} =\left \{ (i,j) \in E |\; i\sim^* j \right\}.
$$

Let $G'=(V,E') \in \mathcal S(G)$, a {\it path} in $G'$ is a
finite sequence of edges $(i_1,i_2),(i_2,i_3),\dots,(i_p,i_{p+1})$
such that $(i_k,i_{k+1})\in E'$  for all $k\in \{1,\dots,p\}$. Two
vertices $i$, $j\in V$ are {\it connected} in $G'$ if there exists
a path in $G'$ joining $i$ and $j$ (i.e. $i_1=i$ and $j_p=j$).
A subset of agents $I\subseteq V$ is a {\it connected component}
of $G'$ if for all $i$, $j\in I$ with $i\ne j$, $i$ and $j$ are
connected in $G'$ and for all $i\in I$, for all $j\in V\setminus
I$, $i$ and $j$ are not connected in $G'$. The set of connected
components of $G'$ is denoted $\mathcal K(G')$. Let us remark that
$\mathcal K(G')$ is a partition of $V$.

We define the vectors of opinions $x(t) =
(x_1(t),\dots,x_n(t))^\top$ and of initial opinions
$x^0=(x_1^0,\dots,x_n^0)^\top$. The dynamics of the vector of
opinions is then given by
$$
x(t+1)=P(t) x(t)
$$
where $P(t)$ is the row stochastic matrix with entries $p_{ij}(t)$.
For a set of agents $I \subseteq V$, with $I=\{v_1,\dots,v_k\}$,
we define the vector of opinions $x_{I}(t)=(x_{v_1}(t),\dots,x_{v_k}(t))^\top$.
Given a $n\times n$ matrix $A$ with entries $a_{ij}$, we define the $k\times k$ matrix $A_{I}$ whose entries are the
$a_{v_{i}v_{j}}$.  In particular, $P_I(t)$ is the matrix with entries $p_{v_{i}v_{j}}(t)$.
Let us remark that $P_{I}(t)$ is generally not row stochastic. 
However, if $I \subseteq V$ is a subset of agents such that no agent in
$I$ is connected to an agent in $V\setminus I$ in the graph
$G(t)$, then it is easy to see that
$$
x_I(t+1)=P_I(t) x_I(t)
$$
and $P_I(t)$ is an aperiodic row stochastic matrix. Moreover, if $I$ is a connected component of $G(t)$ then $P_I(t)$ is irreducible.

The following sections are devoted to the analysis of the community structure of the network of agents.

\section{Asymptotic Connectivity and Agreement}

In this section, we explore the relation between communities and asymptotically connected components of the network.
Let us remark that the set of edges $E$ can be classified into two subsets as follows:
$$
E^f=\left\{(i,j)\in E|\; \exists t_{ij}\in \N,\; \forall s\ge t_{ij},\; (i,j)\notin E(s)\right\}
$$
and
$$
E^\infty = \left\{(i,j)\in E|\; \forall t\in \N,\; \exists s\ge t,\;  (i,j)\in E(s)\right\}.
$$
Intuitively, an edge $(i,j)$ is in $E^f$ if the agents $i$ and $j$
stop interacting with each other in finite time. $E^\infty$
consists of the interactions between agents that are infinitely
recurrent. It is clear that $E^f \cap E^\infty=\emptyset$ and
$E=E^f \cup E^\infty$. Also, since $E$ and thus $E^f$ is a finite
set, there exists $T\in \N$ such that
\begin{equation}
\label{eq:T}
\forall (i,j)\in E^f,\; \forall s\ge T,\; (i,j)\notin E(s).
\end{equation}
Let us remark that the sets $E^f$ and $E^\infty$ and the natural
number $T$ generally depend on the vector of initial opinions
$x^0$. We define the graph $G^\infty=(V,E^\infty)$.
\begin{definition} Let $i,j \in V$, we say that agents $i$ and $j$ are {\it asymptotically connected} if and only if $i$ and $j$ are connected in $G^\infty$.
We say that they are {\it asymptotically disconnected} if they are not asymptotically connected.
\end{definition}

\subsection{Asymptotic Connectivity Implies Asymptotic Agreement}

\begin{proposition}\label{pro:con} Under Assumption~\ref{assum1} (Stochasticity), if two agents $i,j \in V$ are asymptotically
connected then they asymptotically agree.
\end{proposition}

\begin{proof}
Suppose $(i,j)\in E^\infty$. From the definition of $ E^\infty$ there exists a strictly increasing sequence of non-negative integers $(\tau_k)_{k\in\N}$ such that for all $k\in \N$, $(i,j)\in E(\tau_k)$. Then, for all $k\in \N$,
$|x_i(\tau_k)-x_j(\tau_k)|\le R\rho^{\tau_k}$.
Since $\rho\in (0,1)$ and $\displaystyle\lim_{k\rightarrow\infty}\tau_k=+\infty$ and one gets $\displaystyle\lim_{k\rightarrow\infty}x_i(\tau_k)=\displaystyle\lim_{k\rightarrow\infty}x_j(\tau_k)$. On the other hand, the sequences $x_i(t)$ and $x_j(t)$ are convergent, which ensures that

\[
x_i^*=\lim_{t\to\infty}x_i(t)=\lim_{k\rightarrow\infty}x_i(\tau_k)=\lim_{k\rightarrow\infty}x_j(\tau_k)=\lim_{t\to\infty}x_j(t)=x_j^*
\]
The result in the proposition then follows from the transitivity of equality and the
definition of asymptotic connectivity.
\end{proof}

\begin{remark}The notion of asymptotic connectivity has
already been considered in several works (including~\cite{jadbabaie2003,blondel2005,moreau2005})
for proving consensus in multi-agent systems. Actually, the
previous proposition could be proved using Theorem~3
in~\cite{moreau2005}. However, for the sake of self-containment,
we preferred to provide a simpler proof of the result that uses
the specificities of our model.  
\end{remark}

\subsection{Asymptotic Agreement Implies Asymptotic Connectivity}
The converse result of Proposition~\ref{pro:con} is much more
challenging: it is clear that it cannot hold for all initial
conditions. Indeed, if all the initial opinions $x_i^0$ are
identical, then it is clear that the agents asymptotically agree
even though some of them may be asymptotically disconnected which would be the case if the graph $G$ is not connected.
Therefore, we shall prove that the converse result holds
for almost all initial conditions. 
In this paragraph, we will
need additional assumptions in order to be able to prove this
result. The first one is the following:
\begin{assumption}[Invertibility and graph to matrix mapping]
\label{assum2} The sequence of matrices $P(t)$ satisfy the following conditions:
\begin{enumerate}[(a)]
\item For all $t\in \N$, $P(t)$ is invertible.

\item For all $t\in \N$, $t' \in \N$, if $G(t')=G(t)$ then $P(t')=P(t)$.
\end{enumerate}
\end{assumption}
The first assumption is quite strong and we notice that it is not
verified by the original Krause model. However, it can be
enforced, for instance, by choosing $p_{ii}(t) > 1/2$ for all
$i\in V$, for all $t\in \N$, in that case $P(t)$ is a strictly
diagonally dominant matrix and therefore it is invertible. The
second assumption states that $P(t)$ only depends on the graph
$G(t)$, then we shall write $P(t)=P(G(t))$ where $P(G')$ is the
matrix associated to a graph $G'\in \mathcal S(G)$. From the first
assumption, $P(G')$ must be invertible. Then, we can define for
all $t \in \N$, the following set of matrices:
\begin{equation}
\label{eq:qt}
\mathcal Q_t = \left\{ P(G_0)^{-1} P(G_1)^{-1} \dots P(G_{t-1})^{-1}  |\; G_k \in \mathcal S(G),\; 0\le k\le t-1  \right\}.
\end{equation}
Let us remark that since $\mathcal S(G)$ is finite, the set $\mathcal Q_t$ is finite: it has at most $2^{t\times|E|/2}$ elements.

We shall now prove the converse result  of Proposition~\ref{pro:con} in two different cases.

\subsubsection{Average preserving dynamics} We first assume that the opinion dynamics preserves the average of the opinions:

\begin{assumption}[Average preserving dynamics]
\label{assum3} For all $t\in \N$, for all $j\in V$,
$\sum_{i=1}^{n} p_{ij}(t) =1$.
\end{assumption}

This assumption simply means that the matrix $P(t)$ is doubly
stochastic. It is therefore average preserving: the average of
$x(t)$ is equal to the average of $x(t+1)$. Also, if
$I \subseteq V$ is a subset of agents such that no
agent in $I$ is connected to an agent in $V\setminus I$ in the graph $G(t)$, it is easy to show that $P_I(G(t))$ is average preserving.

We now state the main result of the section:
\begin{theorem}\label{th:co} If the matrices $P(t)$ satisfy Assumptions~\ref{assum1} (Stochasticity),~\ref{assum2} (Invertibility and graph to matrix mapping) and~\ref{assum3} (Average preserving dynamics), for almost all vectors of initial opinions $x^0$,
two agents $i,j \in V$ asymptotically agree if and only if they are asymptotically connected.
\end{theorem}

\begin{proof} The if part of the theorem is a consequence of Proposition~\ref{pro:con}. To prove the only if part, let us define the following set
$$
{\mathcal W} = \left\{ (I,J) |\; (I\subseteq V) \land (I\ne \emptyset) \land (J\subseteq V) \land (J\ne \emptyset) \land(I\cap J = \emptyset) \right\}.
$$
Since $V$ is a finite set, it is clear that ${\mathcal W}$ is finite (it has less than $2^{2n}$ elements). For all $(I,J)\in {\mathcal W}$, let $|I|$ and $|J|$ denote
the number of elements of $I$ and $J$ respectively.
We define the vector of $\R^n$, $c_{IJ}$ whose coordinates $c_{IJ,k}=1/|I|$ if $k\in I$, $c_{IJ,k}=-1/|J|$ if $k\in J$, and $c_{IJ,k}=0$ otherwise.
We define the $(n-1)$-dimensional subspace of $\R^n$:
$$
H_{IJ} = \left\{ x\in \R^n |\; c_{IJ}\cdot x = \sum_{i\in I}x_i/|I| - \sum_{j\in J}x_j/|J| =0 \right\}.
$$
Finally, let us define the subset of $\R^n$:
$$
X^0 = \bigcup_{t\in \N}\left( \bigcup_{(I,J)\in \mathcal W} \left( \bigcup_{Q \in \mathcal Q_t} Q H_{IJ} \right) \right)
$$
where $\mathcal Q_t$ is the set of matrices defined in (\ref{eq:qt}).
Since $\mathcal W$ is a finite set and for all $t\in \N$, $\mathcal Q_t$ are finite sets, $X^0$ is a countable union of $(n-1)$-dimensional subspaces of $\R^n$.
Therefore $X^0$ has Lebesgue measure $0$.

Let $x^0\in \R^n$ be a vector of initial opinions, let us assume
that there exist two agents $i,j \in V$ that asymptotically agree
but are asymptotically disconnected. Let us show that necessarily,
$x^0$ belongs to the set $X^0$. Let $I$ and $J$ denote the
connected components of $G^\infty$ containing $i$ and $j$
respectively. Since $i$ and $j$ are asymptotically disconnected,
$I \cap J= \emptyset$, therefore $(I,J) \in \mathcal W$.  Let $T$
be defined as in equation (\ref{eq:T}) (i.e. $E(t)\subseteq
E^{\infty},\ \forall t \geq T$), since no agent in $I$ is connected
to an agent outside of $I$ in $G^\infty$ (and hence in $G(t)$ for
$t\ge T$), we have that for all $t \ge T$,
$x_I(t+1)=P_I(G(t))x_I(t)$. Moreover, $P_I(G(t))$ is average preserving. 
Therefore, for all $t\ge
T$, the average of $x_I(t)$ is the same as the average of
$x_I(T)$. From Proposition~\ref{pro:con}, all agents in $I$
asymptotically agree, then the limit value is necessarily the
average of $x_I(T)$. Therefore $x_i^*=({\bf 1}_{|I|} \cdot
x_I(T))/|I|$ where ${\bf 1}_{|I|}$ denote the $|I|$-dimensional
vector with all entries equal to $1$. A similar discussion gives
that $x_j^*=  ({\bf 1}_{|J|}\cdot x_J(T))/|J|$. Since $i$ and $j$
asymptotically agree, we have $({\bf 1}_{|I|} \cdot x_I(T))/|I| =
({\bf 1}_{|J|}\cdot x_J(T))/|J|$. This means that $x(T) \in
H_{IJ}$ and therefore
$$
x^0 = P(G(0))^{-1} P(G(1))^{-1} \dots P(G(T-1))^{-1} x(T) \in \bigcup_{Q \in \mathcal Q_T} Q H_{IJ}
$$
which leads to $x^0\in X^0$.
\end{proof}

Hence, in the case of average preserving dynamics, asymptotic connectivity is equivalent to asymptotic agreement for almost all vectors of initial opinions.
We shall now prove a similar result under different assumptions.

\subsubsection{Fast convergence assumption} We now replace the average preserving assumption by another assumption.
From Proposition~\ref{pro:conv}, we know that the opinion of each agent converges to its limit value no slower than $O(\rho^t)$. This is an upper bound, numerical experiments show that in practice the convergence to the limit value is often slightly faster than $O(\rho^t)$. This observation motivates
the following assumption.
\begin{assumption}[Fast convergence]
\label{assum4} There exists $\underline{\rho} <\rho$ and $M \ge 0$ such that for all $i\in V$, for all $t \in \N$,
$$
|x_i(t)-x_i^*| \le M \underline{\rho}^t.
$$
\end{assumption}

\begin{remark}\label{rem:assum4} The previous assumption always holds unless there exists $i\in V$ such that
$$
\limsup_{t\rightarrow +\infty} \frac{1}{t} \log(|x_i(t)-x_i^*|) = \log(\rho).
$$
It should be noted that unlike Assumptions~\ref{assum1}
(Stochasticity),~\ref{assum2} (Invertibility and graph to matrix mapping) and~\ref{assum3}
(Average preserving dynamics), it is generally not possible to check a
priori whether Assumption~\ref{assum4} holds.
However, numerical experiments tend to show that in practice,
it does.
\end{remark}

The previous assumption allows us to state the following result:
\begin{lemma}\label{lemma} Under Assumptions~\ref{assum1} (Stochasticity) and~\ref{assum4} (Fast convergence), there exists $T' \in \N$ such that for all $t\ge T'$, $G(t)=G^\infty$.
Moreover, $G^\infty=G_{\mathscr{C}}$.
\end{lemma}

\begin{proof} We shall prove the lemma by showing that there exists $T'\in \N$ such that for all $t\ge T' $, $E(t)\subseteq E^\infty \subseteq E_{\mathscr{C}} \subseteq E(t)$.
Firstly, let $T_1 \ge T$ where $T$ is defined as in equation
(\ref{eq:T}), then for all $t\ge T_1$, $E(t)\subseteq E^\infty$.
Secondly, let $(i,j)\in E^\infty$, then agents $i$ and $j$ are
asymptotically connected. From Proposition~\ref{pro:con}, it
follows that $i$ and $j$ asymptotically agree. Therefore,
$(i,j)\in E_{\mathscr{C}}$. Thirdly, let $(i,j) \in
E_{\mathscr{C}}$, then $x_i^*=x_j^*$ and for all $t\in \N$
\begin{eqnarray*}
|x_i(t) -x_j(t) | & \le & |x_i(t) -x_i^*| + |x_i^* -x_j^*| + |x_j(t) -x_j^*| \\
                  & \le & |x_i(t) -x_i^*| + |x_j(t) -x_j^*|
\end{eqnarray*}
From Assumption~\ref{assum4}, we have for all $t\in \N$,
$$
|x_i(t) -x_j(t) | \le 2 M \underline{\rho}^t.
$$
Since $\underline{\rho}<\rho$, there exists $T_2 \in \N$, such
that for all $t\ge T_2$, $2 M \underline{\rho}^t \le R \rho^t$.
Then, for all $t\ge T_2$, $(i,j)\in E(t)$. Let $T'=\max(T_1,T_2)$,
then for all $t\ge T'$, $E(t)=E^\infty=E_{\mathscr{C}}$ and thus
$G(t)=G^\infty=G_{\mathscr{C}}$.
\end{proof}

The previous result states that after a finite number of steps, the graph of interactions between agents remains always the same. Then, we can state a result similar to Theorem~\ref{th:co}:
\begin{theorem}\label{th:co1} Under Assumptions~\ref{assum1} (Stochasticity),~\ref{assum2} (Invertibility and graph to matrix mapping) and~\ref{assum4} (Fast convergence), for almost all vectors of initial opinions $x^0$,
two agents $i,j \in V$ asymptotically agree if and only if they are asymptotically connected.
\end{theorem}

\begin{proof} The if part of the theorem is a consequence of Proposition~\ref{pro:con}.
To prove the only if part, let us define the following set associated to a spanning subgraph $G'\in \mathcal S(G)$:
$$
{\mathcal W}(G') = \left\{ (I,J)|\; (I\subseteq V) \land (J\subseteq V) \land (I \ne J) \land (I\in \mathcal K(G')) \land (J\in \mathcal K(G')) \right\}.
$$
Since $V$ is a finite set, it is clear that ${\mathcal W}(G')$ is
finite (it has less than $2^{2n}$ elements). Let $(I,J)\in
{\mathcal W}(G')$, $I=\left\{v_1,\dots,v_{|I|} \right\}$,
$J=\left\{w_1,\dots,w_{|J|} \right\}$. Since $I$ and $J$ are
connected components of $G'$, we have that $P_I(G')$ and $P_J(G')$
are aperiodic irreducible row stochastic matrices.
Let $e_I(G')$ and $e_J(G')$ be the left Perron eigenvectors of $P_I(G')$ and $P_J(G')$, respectively:
$$
e_I(G')^\top P_I(G')= e_I(G')^\top \text{ and } e_I(G')\cdot {\bf 1}_{|I|} =1
$$
and
$$
e_J(G')^\top P_J(G')= e_J(G')^\top \text{ and } e_J(G')\cdot {\bf 1}_{|J|} =1.
$$
We define the vector of $\R^n$, $c_{IJ}$ whose coordinates are given by $c_{IJ,v_k}=e_{I,k}$ if $v_k\in I$, $c_{IJ,w_k}=-e_{J,k}$ if $w_k\in J$ and $c_{IJ,k}=0$ if $k\in V\setminus(I\cup J)$.
We define the $(n-1)$-dimensional subspace of $\R^n$:
$$
H_{IJ}(G') = \left\{ x\in \R^n |\; c_{IJ}(G')\cdot x =0 \right\}.
$$
Finally, let us define the subset of $\R^n$:
\begin{equation}
\label{eq:X0}
X^0 = \bigcup_{t\in \N}\left(
\bigcup_{G'\in \mathcal S(G)} \left(
 \bigcup_{(I,J)\in \mathcal W(G')} \left( \bigcup_{Q \in \mathcal Q_t} Q H_{IJ}(G') \right)\right) \right)
\end{equation}
where $\mathcal Q_t$ is the set of matrices defined in
(\ref{eq:qt}). $\mathcal S(G)$ is a finite set and for all $G'\in
\mathcal S(G)$, $\mathcal W(G')$ is a finite set. Moreover for all
$t\in \N$, $\mathcal Q_t$ is a finite set. Then, $X^0$ is a
countable union of $(n-1)$-dimensional subspaces of $\R^n$.
Therefore $X^0$ has Lebesgue measure $0$.

Let $x^0\in \R^n$ be a vector of initial opinions, let us assume
that there exist two agents $i,j \in V$ that asymptotically agree
but are asymptotically disconnected. Let us show that necessarily,
$x^0$ belongs to the set $X^0$. Let $I$ and $J$ denote the
connected components of $G^\infty$ containing $i$ and $j$
respectively. Since $i$ and $j$ are asymptotically disconnected,
$I \ne J$, therefore $(I,J) \in \mathcal W(G^\infty)$. Since $I$
is a connected component of $G^\infty$, it follows from
Lemma~\ref{lemma} that for all $t\ge T'$,
$x_I(t+1)=P_I(G^\infty)x_I(t)$. Moreover, $P_I(G^\infty)$ is an
aperiodic irreducible row stochastic matrix and from the
Perron-Frobenius Theorem (see e.g.~\cite{seneta1981}), it follows
that $1$ is a simple eigenvalue of $P_I(G^\infty)$ and all other
eigenvalues of $P_I(G^\infty)$ have modulus strictly smaller than
$1$. Therefore,
$$
\lim_{t\rightarrow +\infty} x_I(t) = (e_I(G^\infty)\cdot x_I(T')) {\bf 1}_{|I|}
$$
and $x_i^*=e_I(G^\infty)\cdot x_I(T')$.
A similar discussion gives that $x_j^*=e_J(G^\infty)\cdot x_J(T')$.
Since $i$ and $j$ asymptotically agree, we have $e_I(G^\infty)\cdot x_I(T')=e_J(G^\infty)\cdot x_J(T')$.
This means that $x(T')\in H_{I,J}(G^\infty)$ and therefore
$$
x^0 = P(G(0))^{-1} P(G(1))^{-1} \dots P(G(T'-1))^{-1} x(T') \in \bigcup_{Q \in \mathcal Q_{T'}} Q H_{IJ}(G^\infty)
$$
which leads to $x^0\in X^0$.
\end{proof}

In this section, we showed that asymptotic connectivity of agents implies asymptotic agreement and that under additional reasonable assumptions these are actually equivalent except
for a set of vectors of initial opinions of Lebesgue measure $0$.
In other words, we can consider almost surely that the communities of agents correspond to the connected components of the graph $G^\infty$. Actually, we are confident that a similar result holds even without Assumptions~\ref{assum3} or~\ref{assum4}.
However, in this case, the set $X^0$ of initial opinions leading to agreement without connectivity is not necessarily a countable union of $(n-1)$-dimensional subspaces, and it can have much more complex geometrical features. Therefore, we leave the generalization of the results presented in this section as future work.

In the following, under Assumptions~\ref{assum1}
(Stochasticity),~\ref{assum2} (Invertibility and graph to matrix mapping) and~\ref{assum4}
(Fast convergence), we show that an algebraic characterization of
communities can be given in terms of eigenvalues of the matrix
associated to the graph of communities $P(G_{\mathscr{C}})$.

\section{Algebraic Characterization of Communities}

Let $G'\in \mathcal S(G)$, let $I\subseteq V$ be a subset of
agents such that no agent in $I$ is connected to an agent in
$V\setminus I$ in the graph $G'$, then $P_I(G')$ is a row
stochastic matrix. Let $\lambda_1(P_I(G')),\dots,
\lambda_{|I|}(P_I(G'))$ denote the eigenvalues of $P_I(G')$ with
$\lambda_1(P_I(G'))=1$ and
$$
|\lambda_1(P_I(G'))|\ge |\lambda_2(P_I(G'))|\ge \dots \ge |\lambda_{|C|}(P_I(G'))|.
$$

Let $C\in \mathscr{C}$, then no agent in $C$ is connected to an
agent in $V\setminus C$ in the graph $G_{\mathscr{C}}$. The
following theorem gives a characterization of the communities in
terms of the eigenvalues $\lambda_2(P_C(G_{\mathscr{C}}))$ for
$C\in \mathscr{C}$.
\begin{theorem}
\label{th:alg} Under Assumptions~\ref{assum1}
(Stochasticity),~\ref{assum2} (Invertibility and graph to matrix mapping) and~\ref{assum4}
(Fast convergence), for almost all vectors of initial opinions
$x^0$, for all communities $C\in {\mathscr{C}}$, such that $|C|\ge
2$,
$$
|\lambda_2(P_C(G_{\mathscr{C}})) |< \rho.
$$
\end{theorem}

\begin{proof} Let us consider a spanning subgraph
$G'\in \mathcal S(G)$, let $I=\left\{v_1,\dots,v_{|I|} \right\}$,
with $|I|\ge 2$, be a connected component of $G'$ then $P_I(G')$
is an aperiodic irreducible row stochastic matrix. Then, from the
Perron-Frobenius Theorem, it follows that $1$ is a simple
eigenvalue of $P_I(G')$. Therefore, $\lambda_2(P_I(G'))\ne 1$. Let
$f_I(G')$ be a left eigenvector of $P_I(G')$ associated to
eigenvalue $\lambda_2(P_I(G'))$.
Let us define the vector of $\R^n$, $c_I(G')$ whose coordinates
are given by $c_{I,v_k}(G')=f_{I,k}(G')$ if $v_k\in I$ and
$c_{I,k}(G')=0$ if $k\in V\setminus I$. We define the
$(n-1)$-dimensional subspace of $\R^n$:
$$
H_I(G')=\left\{ x\in \R^n |\; c_{I}(G')\cdot x =0 \right\}.
$$
Finally, let us define the subset of $\R^n$:
$$
Y^0 = \bigcup_{t\in \N}\left(
\bigcup_{G'\in \mathcal S(G)} \left(
 \bigcup_{I\in \mathcal K(G'),\; |I|\ge 2} \left( \bigcup_{Q \in \mathcal Q_t} Q H_{I}(G') \right)\right) \right)
$$
where $\mathcal Q_t$ is the set of matrices defined in
(\ref{eq:qt}). $\mathcal S(G)$ is a finite set and for all $G'\in
\mathcal S(G)$, $\mathcal K(G')$ is a finite set. Moreover, for
all $t\in \N$, $\mathcal Q_t$ is a finite set. Then, $Y^0$ is a
countable union of $(n-1)$-dimensional subspaces of $\R^n$.
Therefore $Y^0$ has Lebesgue measure $0$.

Let $X^0$ be given as in equation (\ref{eq:X0}), let $x^0\in \R^n
\setminus X^0$ be a vector of initial opinions. Let us assume
there is a community $C\in {\mathscr{C}}$ with $|C|\ge 2$, such
that $|\lambda_2(P_C(G_{\mathscr{C}})) |\ge \rho$. Let us show
that necessarily, $x^0$ belongs to the set $Y^0$. First, since
$x^0\notin X^0$, we have from the proof of Theorem~\ref{th:co1}
that $C$ is a connected component of $G^\infty=G_{\mathscr{C}}$.
Therefore, from Lemma~\ref{lemma}, there exists $T'\in \N$, such that for all $t\ge T'$,
$x_C(t+1)=P_C(G_{\mathscr{C}})x_C(t)$ and $P_C(G_{\mathscr{C}})$
is an aperiodic irreducible row stochastic matrix. From the
Perron-Frobenius Theorem, it follows that $1$ is a simple
eigenvalue of $P_C(G_{\mathscr{C}})$ and all other eigenvalues of
$P_C(G_{\mathscr{C}})$ have modulus strictly smaller than $1$. Let
$e_C(G_{\mathscr{C}})$ be the left Perron eigenvector of
$P_C(G_{\mathscr{C}})$:
$$
e_C(G_{\mathscr{C}})^\top P_C(G_{\mathscr{C}})=
e_C(G_{\mathscr{C}})^\top \text{ and } e_C(G_{\mathscr{C}})\cdot
{\bf 1}_{|C|} =1
$$
Then
$$
\lim_{t\rightarrow +\infty} x_C(t) = x_C^* \; \text{ where }\;
x_C^* = (e_C(G_{\mathscr{C}})\cdot x_C(T')) {\bf 1}_{|C|}.
$$
Let us remark that for all $t\ge T'$,
\begin{equation}
\label{eq:rec} x_C(t+1)-x_C^*=P_C(G_{\mathscr{C}})(x_C(t)-x_C^*).
\end{equation}
Let $f_C(G_{\mathscr{C}})$ be a left eigenvector of
$P_C(G_{\mathscr{C}})$ associated to eigenvalue
$\lambda_2(P_C(G_{\mathscr{C}}))$:
$$f_C(G_{\mathscr{C}})^\top P_C(G_{\mathscr{C}})=\lambda_2(P_C(G_{\mathscr{C}}))f_C(G_{\mathscr{C}})^\top.$$
Then, it follows from equation (\ref{eq:rec}) that for all $t\ge T'$,
$$
f_C(G_{\mathscr{C}})\cdot(x_C(t)-x_C^*)=  f_C(G_{\mathscr{C}})
\cdot(x_C(T')-x_C^*) \lambda_2(P_C(G_{\mathscr{C}}))^{(t-T')}.
$$
Therefore, by the Cauchy-Schwarz inequality, we have for all $t\ge T'$
\begin{eqnarray*}
\|x_C(t)-x_C^*\| & \ge & \frac{|f_C(G_{\mathscr{C}})\cdot(x_C(t)-x_C^*)|}{\|f_C(G_{\mathscr{C}})\|}\\
& \ge & \frac{|f_C(G_{\mathscr{C}})
\cdot(x_C(T')-x_C^*)|}{\|f_C(G_{\mathscr{C}})\|}
|\lambda_2(P_C(G_{\mathscr{C}}))|^{(t-T')}.
\end{eqnarray*}
Since we assumed $|\lambda_2(P_C(G_{\mathscr{C}}))| \ge \rho$, we
have for all $t\ge T'$
\begin{equation}
\label{eq:ineq1} \|x_C(t)-x_C^*\| \ge
\frac{|f_C(G_{\mathscr{C}})\cdot(x_C(T')-x_C^*)|}{\|f_C(G_{\mathscr{C}})\|
\rho^{T'}} \rho^{t}.
\end{equation}
Now, let us remark that it follows from Assumption~\ref{assum4} that for all $t \in \N$
\begin{equation}
\label{eq:ineq2}
\|x_C(t)-x_C^*\| \le \sqrt{|C|} M \underline{\rho}^{t}.
\end{equation}
Inequalities (\ref{eq:ineq1}) and (\ref{eq:ineq2}) give for all $t\ge T'$
$$
\frac{|f_C(G_{\mathscr{C}})\cdot(x_C(T')-x_C^*)|}{\|f_C(G_{\mathscr{C}})\|
\rho^{T'}} \rho^{t} \le \sqrt{|C|} M \underline{\rho}^{t} .
$$
Since $\underline{\rho} < \rho$, the previous inequality holds for
all $t\ge T'$ if and only if
$|f_C(G_{\mathscr{C}})\cdot(x_C(T')-x_C^*)|=0$. Therefore,
$f_C(G_{\mathscr{C}})\cdot x_C(T')=
f_C(G_{\mathscr{C}})\cdot(x_C(T')-x_C^*)=0$ which means that
$x(T')\in H_C(G_{\mathscr{C}})$. Therefore,
$$
x^0 = P(G(0))^{-1} P(G(1))^{-1} \dots P(G(T'-1))^{-1} x(T') \in
\bigcup_{Q \in \mathcal Q_{T'}} Q H_C(G_{\mathscr{C}})
$$
which leads to $x^0\in Y^0$. Therefore, we have proved that for
all vectors of initial opinions $x^0\in \R^{n}\setminus (X^0\cup
Y^0)$,  for all communities $C\in {\mathscr{C}}$ such that
$|C|\ge2$, $|\lambda_2(P_C(G_{\mathscr{C}}))| < \rho$. We conclude
by remarking that $X^0\cup Y^0$ is a set of Lebesgue measure $0$.
\end{proof}

In this section, we showed that the community structure
${\mathscr{C}}$ satisfies some properties related to the
eigenvalues of the matrix $P_{C}(G_{\mathscr{C}})$, for $C\in
\mathscr{C}$. In the following, we use this result to address the
problem of community detection in graphs.

\section{Application: Community Detection in Graphs}

In this section, we propose to use a model of opinion dynamics with decaying confidence to address the problem of community detection in graphs. 

\subsection{The Community Detection Problem} In the usual sense,
communities in a graph are groups of vertices such that the
concentration of edges inside one community is high and the
concentration of edges between communities is comparatively low.
Because of the increasing need of analysis tools for understanding
complex networks in social sciences, biology, engineering or
economics, the community detection problem has attracted a lot of
attention in the recent years. The problem of community detection
is however not rigorously defined mathematically. One reason is
that community structures may appear at different scales in the
graph: there can be communities inside communities. Another reason
is that communities are not necessarily disjoint and can overlap. 
We refer the reader to the excellent
survey~\cite{fortunato2009} and the references therein for more
details.
Some formalizations of the community detection problem have been proposed in
terms of optimization of quality functions such as modularity~\cite{newman2004}
or partition stability~\cite{lambiotte2009}.

\subsubsection{Quality functions}
Modularity has been introduced in~\cite{newman2004}, the modularity of a partition measures how well the partition reflects the community structure of a graph. More precisely, let $G=(V,E)$ be an undirected graph with $E$ symmetric and anti-reflexive. For a vertex $i\in V$ the degree $d_i$ of $i$ is the number of neighbors of $i$ in $G$. Let $\mathcal P$ be a partition of $V$. Essentially, the modularity ${\mathsf Q}(\mathcal P)$ of the partition $\mathcal P$ is the proportion of edges within the classes of the partition minus the expected proportion of such edges, where the expected number of edges between vertex $i$ and $j$ is assumed to be $d_i d_j/|E|$:
$$
{\mathsf Q}(\mathcal P)=\frac{1}{|E|} \sum_{I\in \mathcal P} \sum_{i,j \in I} \left( a_{ij} - \frac{d_i d_j}{|E|} \right)
$$
where $a_{ij}$ are the coefficients of the adjacency matrix of $G$ ($a_{ij}=1$ if $(i,j)\in E$, $a_{ij}=0$ otherwise).
The higher the modularity, the better the partition reflects the community structure of the graph. Thus, it is reasonable to formulate the community detection problem as modularity maximization.
However, it has been shown that this optimization problem is NP-complete~\cite{brandes2008}. 
Therefore, approaches for community detection rely mostly on heuristic methods. 
In~\cite{newman2006}, a modularity optimization algorithm is proposed based on spectral relaxations. 
Using the eigenvectors of the modularity matrix, it is possible to determine a good initial guess of the community structure of the graph. Then, the obtained partition is refined using local combinatorial optimization.
In ~\cite{blondel2008}, a hierarchical combinatorial approach for modularity optimization is presented. This algorithm 
which can be used for very large networks,
is
currently the one that obtains the partitions with highest modularity.

However, modularity has the drawback that it fails to capture communities at different scales.
The notion of partition stability~\cite{lambiotte2009} makes it possible to overcome this limitation.  
Let us consider a continuous-time process associated with a random walk over the graph $G$ where transitions are triggered by a homogeneous Poisson process. Assume that the initial distribution is the stationary distribution.
Then, the stability at time $t\in \R^+$ of the partition $\mathcal P$ is defined as 
$$
{\mathsf R}(\mathcal P,t) = \sum_{I\in \mathcal P} p(I,t)-p(I,\infty)
$$ 
where $p(I,t)$ is the probability for a walker to be in the class $I$ initially and at time $t$. Stability measures the quality of a partition by giving a positive contribution to communities from which a random walker is unlikely to escape within the given time scale $t$.
For small values of $t$, this gives more weights to small communities whereas for larger values of $t$, larger communities are favored. Thus, by searching the partitions maximizing the stability for several values of $t$, one can detect communities at several scales.

\subsubsection{Eigenvalues of the normalized Laplacian matrix}
We give an alternative formulation of the community detection problem using a measure of connectivity of graphs given by the eigenvalues of their {\it normalized Laplacian matrix}.
Let $G=(V,E)$ be an undirected graph with $V=\{1,\dots,n\}$, with $n\ge 2$. For a vertex
$i\in V$, the degree $d_i(G)$ of $i$ is the number of neighbors of $i$ in $G$. The normalized Laplacian of the graph
$G$ is the matrix $L(G)$ given by
$$
L_{ij}(G)=
\left\{
\begin{array}{ll}
1 & \text{if } i=j \text{ and } d_i(G)\ne 0, \\
\frac{-1}{\sqrt{d_i(G) d_j(G)}} & \text{if } (i,j)\in E, \\
0 & \text{otherwise.}
\end{array}
\right.
$$
Let us review some of the properties of the normalized Laplacian matrix (see e.g.~\cite{chung1997}).
$\mu_1(L(G))=0$ is always an eigenvalue of $L(G)$, it is simple if and only if $G$ is connected.
All other eigenvalues are real and belong to the interval $[0,2]$.
The second smallest eigenvalue of the normalized Laplacian matrix is denoted $\mu_2(L(G))$. It can serve as an algebraic measure of the connectivity: $\mu_2(L(G))=0$ if the graph $G$ has two
distinct connected components, $\mu_2(L(G))=n/(n-1)$ if the graph is the complete graph (for all $i,j \in V$, $i\ne j$, $(i,j)\in E$), in the other cases $\mu_2(L(G))\in (0,1]$.

\begin{remark} The second smallest eigenvalue of the (non-normalized) Laplacian matrix is called algebraic connectivity
of a graph. In this paper, we prefer to use the eigenvalues of the normalized Laplacian matrix because
it is less sensitive to the size of the graph. For instance, if $G$ is the complete graph then $\mu_2(L(G))=n/(n-1)$ whereas its algebraic connectivity is $n$.
\end{remark}

Let $\mathcal P$ be a partition of the set of vertices $V$.
For all $I\in {\mathcal P}$, with $|I|\ge 2$, $L(G_I)$ denotes the normalized Laplacian matrix of the graph
$G_I=(I,E_I)$ consisting of the set of vertices $I$ and of the set of edges of $G$ between elements of $I$.
Let us define the following measure associated to the partition $\mathcal P$
$$
\underline{\mu_2}(\mathcal P) = {\min_{I\in \mathcal{P}, |I|\ge 2} \mu_2(L(G_{I}))}.
$$
Essentially, $\underline{\mu_2}(\mathcal P)$ measures the connectivity of the less connected component of $G_{\mathcal P}$.


We now propose a new formulation of the community detection problem:
\begin{problem}\label{prob} Given a graph $G=(V,E)$ and a real number $\delta \in (0,1]$,
find a partition ${\mathcal P}$ of $V$ such that for all $I\in  \mathcal  P$, such that $|I|\ge 2$,
$
\mu_2(L(G_{I})) > \delta
$
(i.e. $\underline{\mu_2}(\mathcal P) >\delta$).
\end{problem}

If $\mu_2(L(G))>\delta$, it is sufficient to choose the trivial partition $\mathcal P=\{V\}$.
If $\delta \ge \mu_2(L(G))$, then we want to find groups of vertices that are more densely connected than the global graph. This coincides with the notion of community. The larger $\delta$ the more densely connected the communities.
This makes it possible to search for communities at different scales of the graph.
Let us remark that Problem~\ref{prob} generally has several solutions. Actually, the trivial partition $\mathcal P=\{\{1\},\dots,\{n\}\}$ is always a solution. 
In the following, we show how non-trivial solutions to Problem~\ref{prob} can be obtained using a model of opinion dynamics with decaying confidence. 
We evaluate the modularity of the partitions we obtain and compare our results to those obtained using modularity optimization algorithms presented in~\cite{newman2006,blondel2008}.

\subsection{Opinion Dynamics for Community Detection} Let $\alpha\in(0,1/2)$, we consider the opinion dynamics with decaying confidence model given by:
\begin{equation}
\label{eq:model2}
x_i(t+1)= \left\{
\begin{array}{ll}
\displaystyle{x_i(t) + \frac{\alpha}{|N_i(t)|} \sum_{j\in N_i(t)} (x_j(t)-x_i(t))} & \text{if } N_i(t)\ne \emptyset \\
x_i(t) & \text{if } N_i(t) = \emptyset
\end{array}
\right.
\end{equation}
where $N_i(t)$ is given by equation~(\ref{eq:neigh2}). It is straightforward to check that this model is a particular case of the model given by equations~(\ref{eq:dynamics}) and~(\ref{eq:neigh}) and that Assumption~\ref{assum1} (Stochasticity) holds.
Moreover, since $\alpha \in (0,1/2)$ it follows that for all $i\in V$, $t\in \N$, $p_{ii}(t)>1/2$. Therefore
the matrix $P(t)$ is strictly diagonally dominant and hence it is invertible. Also, $P(t)=P(G(t))$,
where for a subgraph $G'$,
$P(G')= Id -\alpha Q(G')$ where $Id$ is the identity matrix and
\begin{equation}
\label{eq:Q}
Q_{ij}(G')=\left\{
\begin{array}{ll}
1 & \text{if } i=j \text{ and } d_i(G')\ne 0, \\
\frac{-1}{d_i(G')} & \text{if } (i,j)\in E', \\
0 & \text{otherwise.}
\end{array}
\right.
\end{equation}
where $d_i(G')$ denotes the degree of $i$ in the graph $G'$. Therefore, Assumption~\ref{assum2} (Invertibility and graph to matrix mapping) holds as well.
Let us remark that the matrix $P(t)$ is generally not average preserving and therefore Assumption~\ref{assum3}
does not hold.

Before stating the main result of this section, we need to prove the following lemma~:
\begin{lemma} \label{lem:eig}
Let $\mathcal P$ be a partition of $V$, $I\in \mathcal P$ such that $|I|\ge 2$. Then,
$\lambda$ is an eigenvalue of $P_I(G_{\mathcal P})$ if and only if $\mu=(1-\lambda)/\alpha$ is an eigenvalue of
$L(G_I)$.
\end{lemma}

\begin{proof} First, let us remark that $P_I(G_{\mathcal P})=Id-\alpha Q(G_I)$ where $Q(G_I)$ is defined as in equation~(\ref{eq:Q}). Then, let us introduce the matrices $R(G_I)$ and $D(G_I)$ defined by
$$
R_{ij}(G_I)=\left\{
\begin{array}{ll}
\frac{1}{\sqrt{d_i(G_I)}} & \text{if } i=j \text{ and } d_i(G_I)\ne 0, \\
\frac{-1}{d_i(G_I)\sqrt{d_j(G_I)}} & \text{if } (i,j)\in E_I, \\
0 & \text{otherwise.}
\end{array}
\right.
$$
and
$$
D_{ij}(G_I)=\left\{
\begin{array}{ll}
\sqrt{d_i(G_I)} & \text{if } i=j, \\
0 & \text{otherwise.}
\end{array}
\right.
$$
Let us remark that $L(G_I)=D(G_I)R(G_I)$ and $Q(G_I)=R(G_I)D(G_I)$. It follows that $L(G_I)$ and $Q(G_I)$
have the same eigenvalues. The stated result is then obtained from the fact that the matrix $Q(G_I)=(Id-P_I(G_{\mathcal P}))/\alpha$.
\end{proof}

We now state the main result of the section which is a direct consequence of Theorem~\ref{th:alg} and Lemma~\ref{lem:eig}:
\begin{corollary}\label{corol} Let $\rho=1-\alpha \delta$, under Assumption~\ref{assum4} (Fast convergence), for almost all vectors of initial opinions $x^0$, the set of communities ${\mathscr{C}}$ obtained by the opinion dynamics model~(\ref{eq:model2}) is a solution to Problem~\ref{prob}.
\end{corollary}

\subsection{Examples}

In this section, we propose to evaluate experimentally the validity of our approach on three benchmarks taken from~\cite{newman2006}. 

\subsubsection{Zachary karate club}
We propose to evaluate our approach on a standard benchmark for community detection: the karate club network
initially studied by Zachary in~\cite{zachary1977}. This is a social network with $34$ agents shown on the top left part
of Figure~\ref{fig}. The original study shows the existence of two communities represented on the figure by squares and triangles.

We propose to use our opinion dynamics model (\ref{eq:model2}) to uncover the community structure of this network.
We chose $4$ different values for $\delta$ and $2$ different values for parameters $R$ and $\alpha$. The parameter $\rho$ was chosen according 
to Corollary~\ref{corol}: $\rho=1-\alpha \delta$. 
For each combination of parameter value, the model was simulated
for $1000$ different vectors of initial opinions chosen randomly in $[0,1]^{34}$. Simulations were performed as long as
enabled by floating point arithmetics.

The experimental results are reported in
Table~\ref{tab1}. For each combination of parameter value, we indicate
the partitions in communities that are the most frequently obtained after running the
opinion dynamics model. For each partition ${\mathscr{C}}$, we
give the number of communities in the partition, the measure $\underline{\mu_2}(\mathscr C)$, this value being
greater than $\delta$ indicates that Problem~\ref{prob} has been
solved. We computed the modularity $\mathsf Q(\mathscr{C})$ in order to
evaluate the quality of the obtained partition. We also indicate
the number of times that each partition occurred over the $1000$
simulations of the opinion dynamics model.

We can check in Table~\ref{tab1} that all the partitions are solutions of Problem~\ref{prob}.
Let us remark that in general the computed partition depends on the initial vector of opinions, this is the case for
$\delta=0.3$ and $\delta=0.4$. Also, changing the parameters $R$ and $\alpha$ seems to have some effect on the probability of obtaining a given partition.
For instance, for $\delta=0.3$, the probabilities of obtaining one partition are significantly different for $R=1$ and $R=10$. 
Also, for $\delta=0.4$, the probabilities are slightly different for $\alpha=0.1$ and $\alpha=0.2$. 

\begin{table}[!t]
\begin{center}
\begin{tabular}{|c|c|c|c|c|c|c|c|}
\hline
$\delta$ & $|\mathscr{C}|$ & $\underline{\mu_2}(\mathscr{C})$ & ${\sf Q}(\mathscr{C})$
& \hspace{-0.4cm} \begin{scriptsize} \begin{tabular}{c} Occurences  \\ $R=1$, $\alpha=0.1$ \end{tabular} \end{scriptsize} \hspace{-0.4cm}
& \hspace{-0.4cm} \begin{scriptsize} \begin{tabular}{c} Occurences   \\  $R=10$, $\alpha=0.1$ \end{tabular} \end{scriptsize} \hspace{-0.4cm}
& \hspace{-0.4cm} \begin{scriptsize} \begin{tabular}{c} Occurences \\ $R=1$, $\alpha=0.2 $ \end{tabular} \end{scriptsize} \hspace{-0.4cm}
& \hspace{-0.4cm} \begin{scriptsize} \begin{tabular}{c} Occurences  \\ $R=10$, $\alpha=0.2$ \end{tabular} \end{scriptsize} \hspace{-0.4cm} \\
\hline
\hline
0.1 & 1 & 0.132 & 0 & 1000 & 1000 & 1000 & 1000 \\
\hline
0.2 & 2 & 0.250 & 0.360 & 1000 & 999 & 1000 & 999 \\
\hline
0.3 & 3 & 0.334 & 0.399 & 691 & 105 & 679 & 63 \\
0.3 & 3 & 0.363 & 0.374 & 283 & 891 & 298 & 937 \\
\hline
0.4 & 4 & 0.566 & 0.417 & 924 & 994 & 884 & 897 \\
0.4 & 5 & 0.566 & 0.402 & 15 & 6 & 54 & 98\\
\hline
\end{tabular}
\vspace{0.3cm}
\caption{Properties of the partitions of the karate club network obtained by the opinion dynamics model ($1000$ different vectors of initial opinions for each combination of parameter values).}
\label{tab1}
\end{center}
\vspace{-0.5cm}
\end{table}

However, it is interesting to note that the partitions that are obtained for the same value of parameter $\delta$ 
have modularities of the same order of magnitude which seems to show that these are of comparable quality.
The partition with maximal modularity is obtained for $\delta=0.4$, it is a partition in 4 communities with modularity $0.417$. 
As a comparison, algorithms~\cite{newman2006,blondel2008} obtain a partition in 4 communities with modularity $0.419$.
This shows that our approach not only allows to solve Problem~\ref{prob} but also furnishes partitions with a good modularity which might seem surprising given the fact
that our approach, contrarily to~\cite{newman2006,blondel2008} does not try to maximize modularity. 

\begin{figure}[!t]
\begin{center}
\vspace{-6cm}
\begin{tabular}{cc}
\hspace{-2.5cm}
\includegraphics[angle=0,scale=0.6]{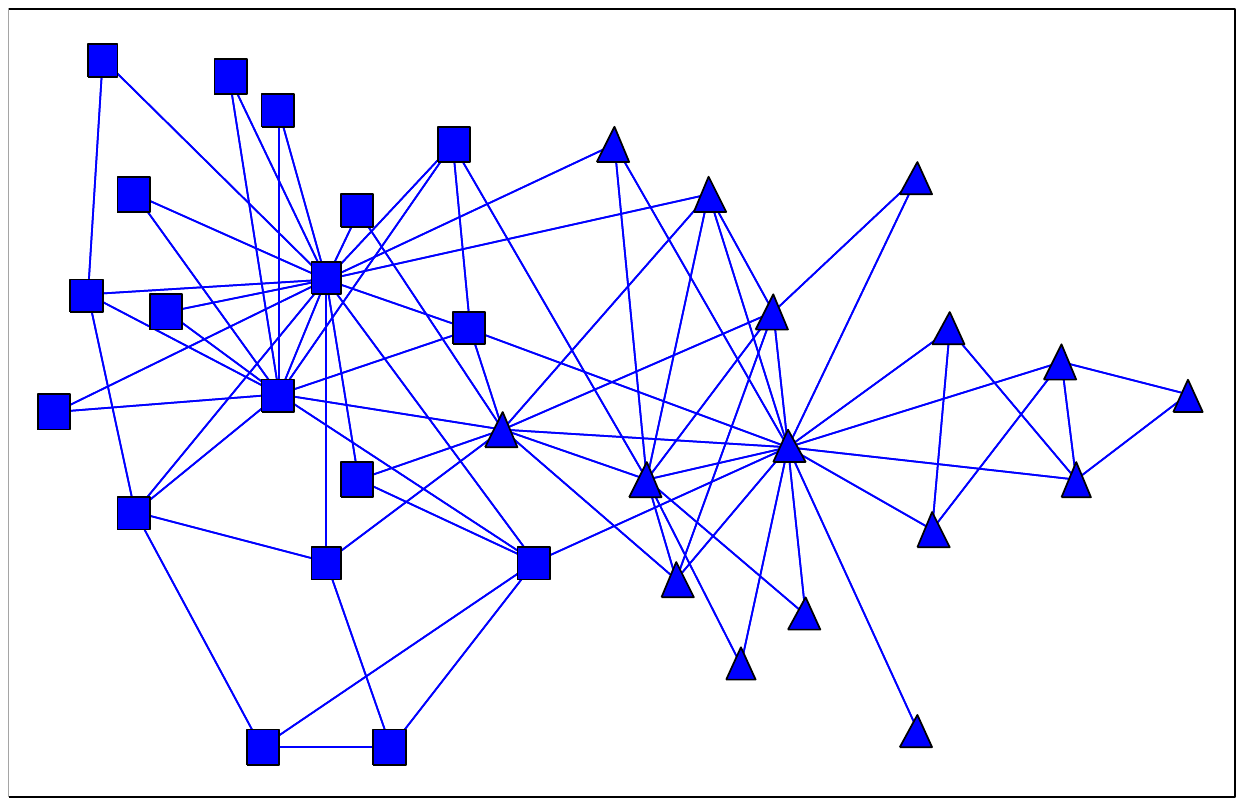}&
\hspace{-5cm}
\includegraphics[angle=0,scale=0.6]{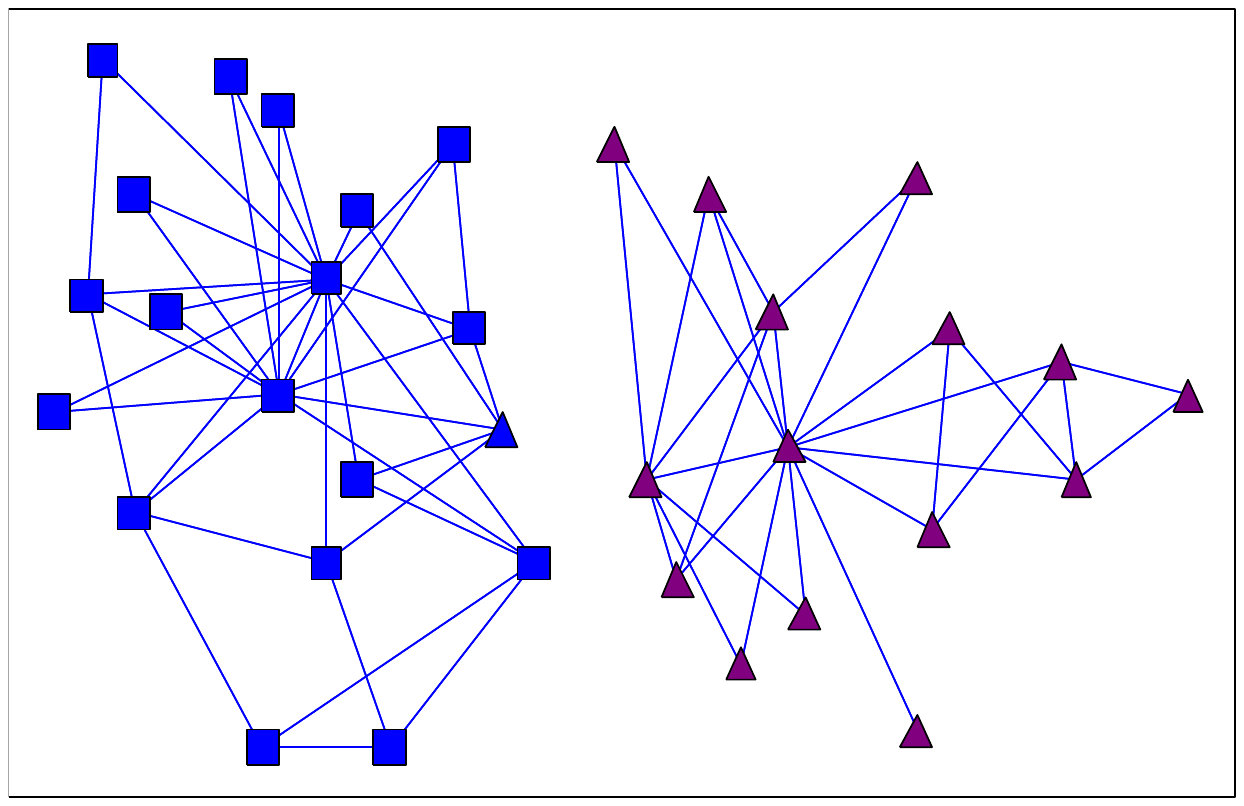}\\
\vspace{-13.5cm}
\\
\hspace{-2.5cm}
\includegraphics[angle=0,scale=0.6]{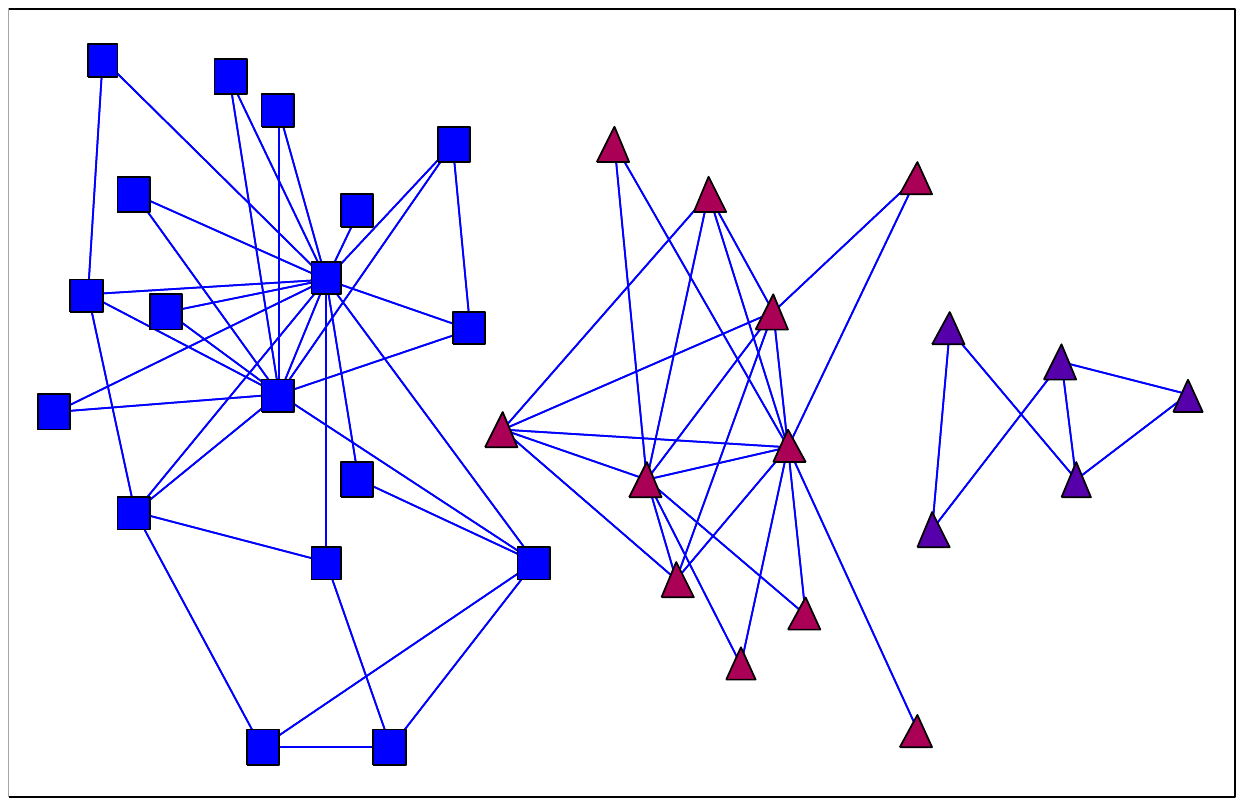}&
\hspace{-5cm}
\includegraphics[angle=0,scale=0.6]{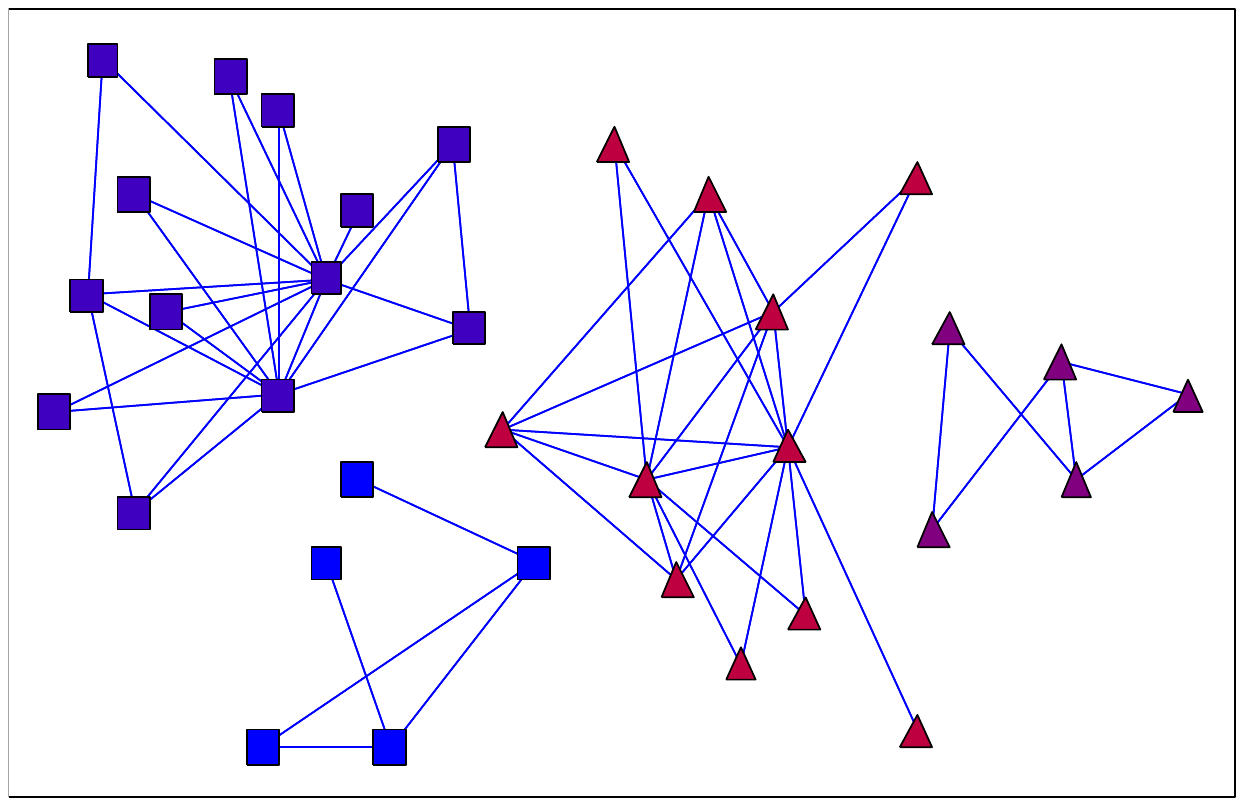}
\end{tabular}
\vspace{-7.cm} \caption{Graphs $G_{\mathscr{C}}$ for the most
frequently obtained partition of the karate club network for $R=1$, $\alpha=0.1$ and
$\delta=0.1$ (top left), $\delta=0.2$ (top right), $\delta=0.3$
(bottom left), $\delta=0.4$ (bottom right).} \label{fig}
\end{center}
\end{figure}

\begin{figure}[!t]
\begin{center}
\includegraphics[angle=0,scale=0.55]{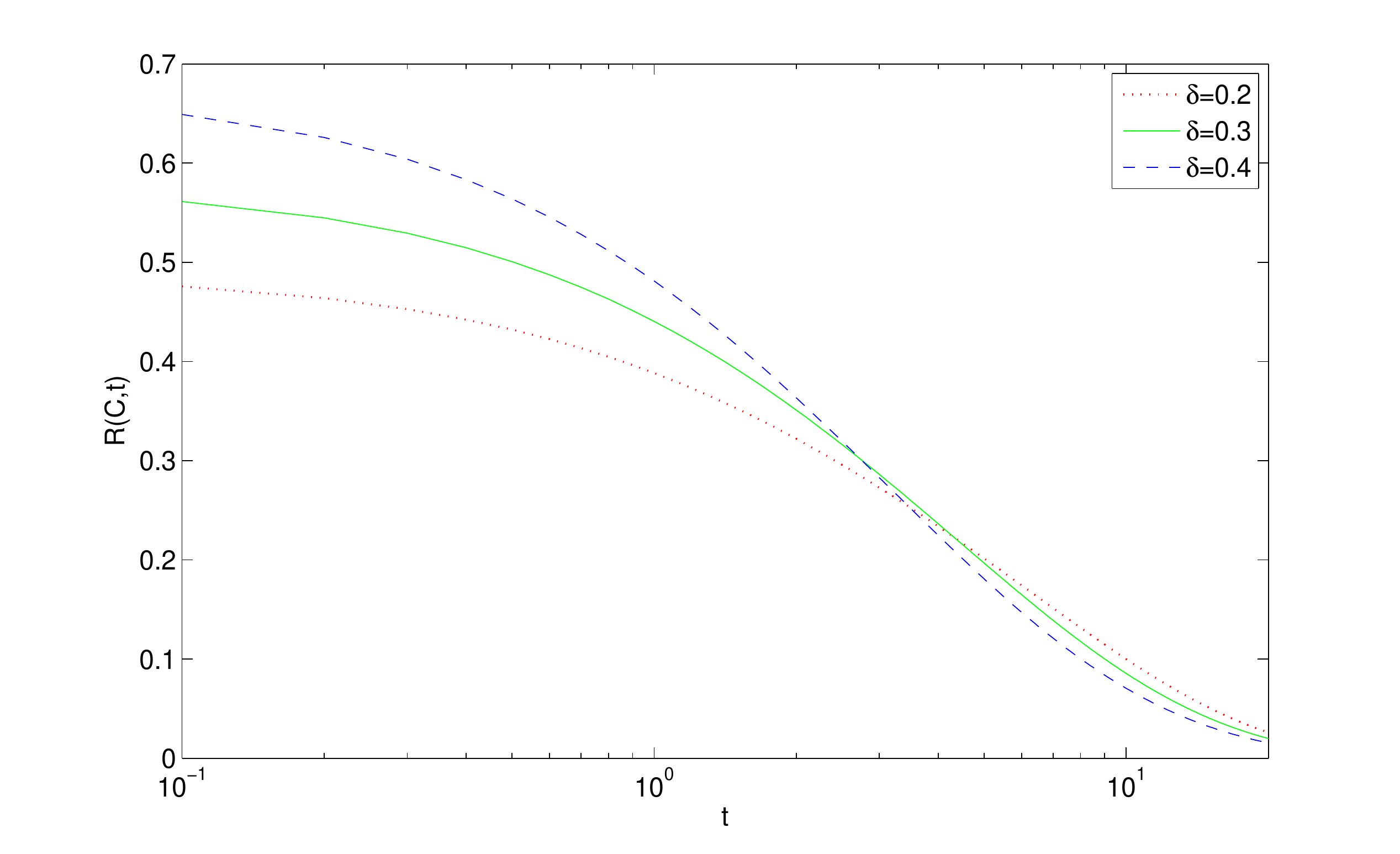}
\vspace{-1cm}
\caption{Stability of the partitions presented in Figure~\ref{fig}.} \label{fig1a}
\end{center}
\end{figure}

In Figure~\ref{fig}, we represented the graphs of communities
$G_{\mathscr{C}}$ that are the most frequently obtained for $R=1$, $\alpha=0.1$ and the
different values of $\delta$. It is interesting to remark that for
$\delta=0.2$ we almost obtained the communities that were reported
in the original study~\cite{zachary1977}. Only one agent has been
classified differently. One may argue that this agent has
originally $4$ neighbors in each community so it could be
classified in one or the other. It is also interesting to see that
our approach allows us to search for communities at different
scales of the graph. When $\delta$ increases, the communities
become smaller but more densely connected. This is corroborated by computing the stability of these partitions (see Figure~\ref{fig1a}).
We can see that the partition with maximal stability changes according to time-scale $t$: for small values of $t$ the partition in $4$ communities is better,
for intermediate values of $t$ the partition in $3$ communities has the largest stability, for large values of $t$ the partition in $2$ communities maximizes the stability.

\subsubsection{Books on American Politics} We propose to use our approach on an example consisting of a network of $105$ books on politics~\cite{newman2006},
initially compiled by V. Krebs (unpublished, see \verb+www.orgnet.com+). In this network, each vertex represents a book on American politics bought from Amazon.com.
An edge between two vertices means that these books are frequently purchased by the same buyer. The network is presented on  the top left part
of Figure~\ref{fig2} where the shape of the vertices represent the political alignement of the book (liberal, conservative, centrist).

We used our opinion dynamics model (\ref{eq:model2}) to uncover the community structure of this network.
We chose $3$ different values for $\delta$ and $2$ different values for parameters $R$ and $\alpha$. The parameter $\rho$ was chosen according 
to Corollary~\ref{corol}: $\rho=1-\alpha \delta$. 
For each combination of parameter value, the model was simulated
for $1000$ different vectors of initial opinions chosen randomly in $[0,1]^{105}$. Simulations were performed as long as
enabled by floating point arithmetics.
The experimental results are reported in Table~\ref{tab2}.

\begin{table}[!hb]
\begin{center}
\begin{tabular}{|c|c|c|c|c|c|c|c|}
\hline
$\delta$ & $|\mathscr{C}|$ & $\underline{\mu_2}(\mathscr{C})$ & ${\sf Q}(\mathscr{C})$
& \hspace{-0.4cm} \begin{scriptsize} \begin{tabular}{c} Occurences  \\ $R=1$, $\alpha=0.1$ \end{tabular} \end{scriptsize} \hspace{-0.4cm}
& \hspace{-0.4cm} \begin{scriptsize} \begin{tabular}{c} Occurences   \\  $R=10$, $\alpha=0.1$ \end{tabular} \end{scriptsize} \hspace{-0.4cm}
& \hspace{-0.4cm} \begin{scriptsize} \begin{tabular}{c} Occurences \\ $R=1$, $\alpha=0.2 $ \end{tabular} \end{scriptsize} \hspace{-0.4cm}
& \hspace{-0.4cm} \begin{scriptsize} \begin{tabular}{c} Occurences  \\ $R=10$, $\alpha=0.2$ \end{tabular} \end{scriptsize} \hspace{-0.4cm} \\
\hline
\hline
0.1 & 2 & 0.134 & 0.457 & 980 & 1000 & 640 & 581 \\
0.1 & 2 & 0.129 & 0.457 & 20 &  0    & 360 & 419 \\
\hline
0.15 & 3 & 0.182 & 0.499 & 898 & 1000 & 905 & 1000 \\
0.15 & 3 & 0.187 & 0.494 & 102 & 0 & 95 & 0 \\
\hline
0.2 & 4 & 0.269 & 0.523 & 678 & 1000 & 673 & 1000 \\
0.2 & 4 & 0.266 & 0.512 & 218 & 0    & 207 & 0 \\
0.2 & 4 & 0.269 & 0.520 & 49 & 0    & 72 & 0 \\
\hline
\end{tabular}
\vspace{0.3cm}
\caption{Properties of the partitions of the books network obtained by the opinion dynamics model ($1000$ different vectors of initial opinions for each combination of parameter values).}
\label{tab2}
\end{center}
\vspace{-0.5cm}
\end{table}

\begin{figure}[!t]
\begin{center}
\vspace{-4.8cm}
\begin{tabular}{cc}
\hspace{-1.cm}
\includegraphics[angle=0,scale=0.5]{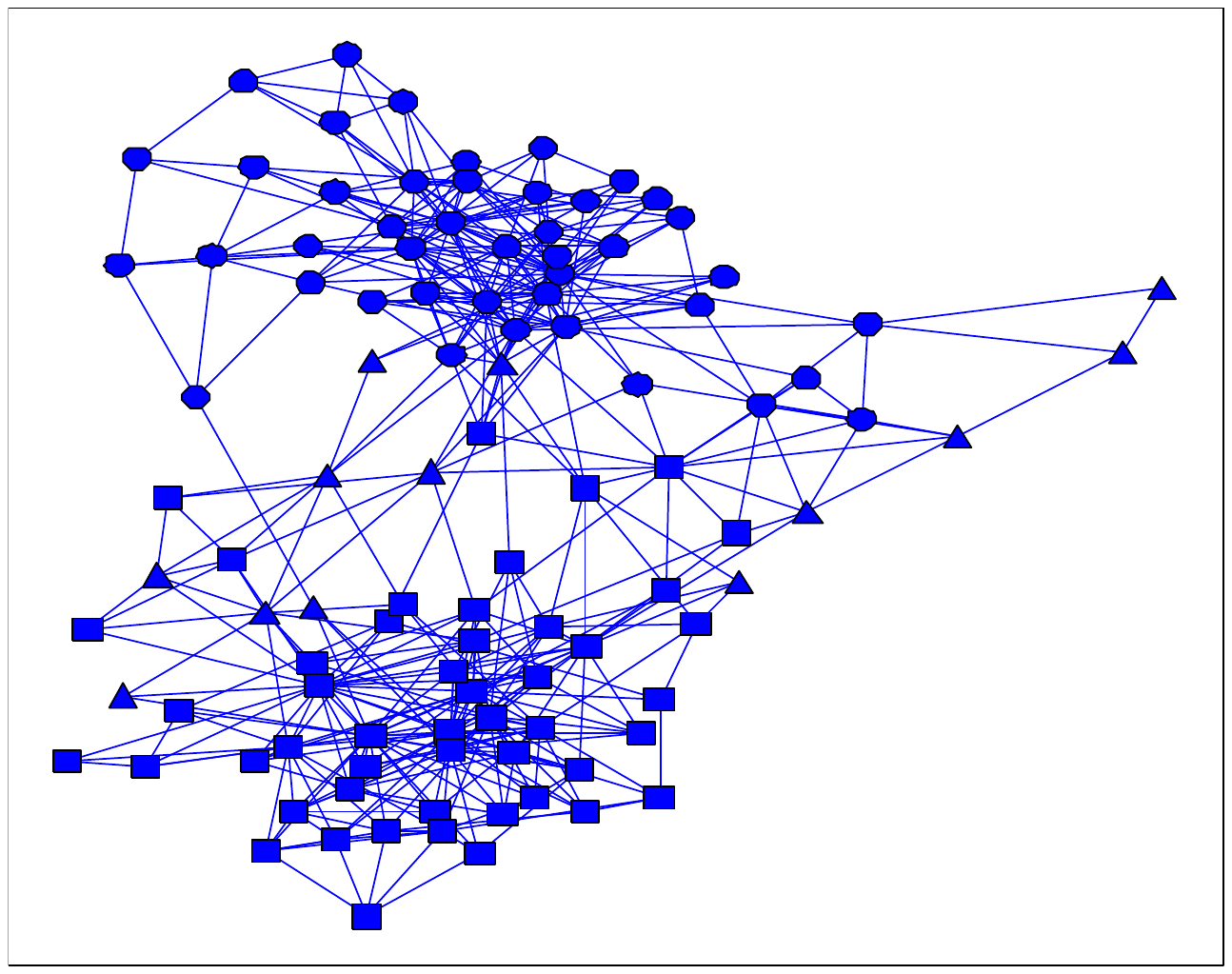}&
\hspace{-4cm}
\includegraphics[angle=0,scale=0.5]{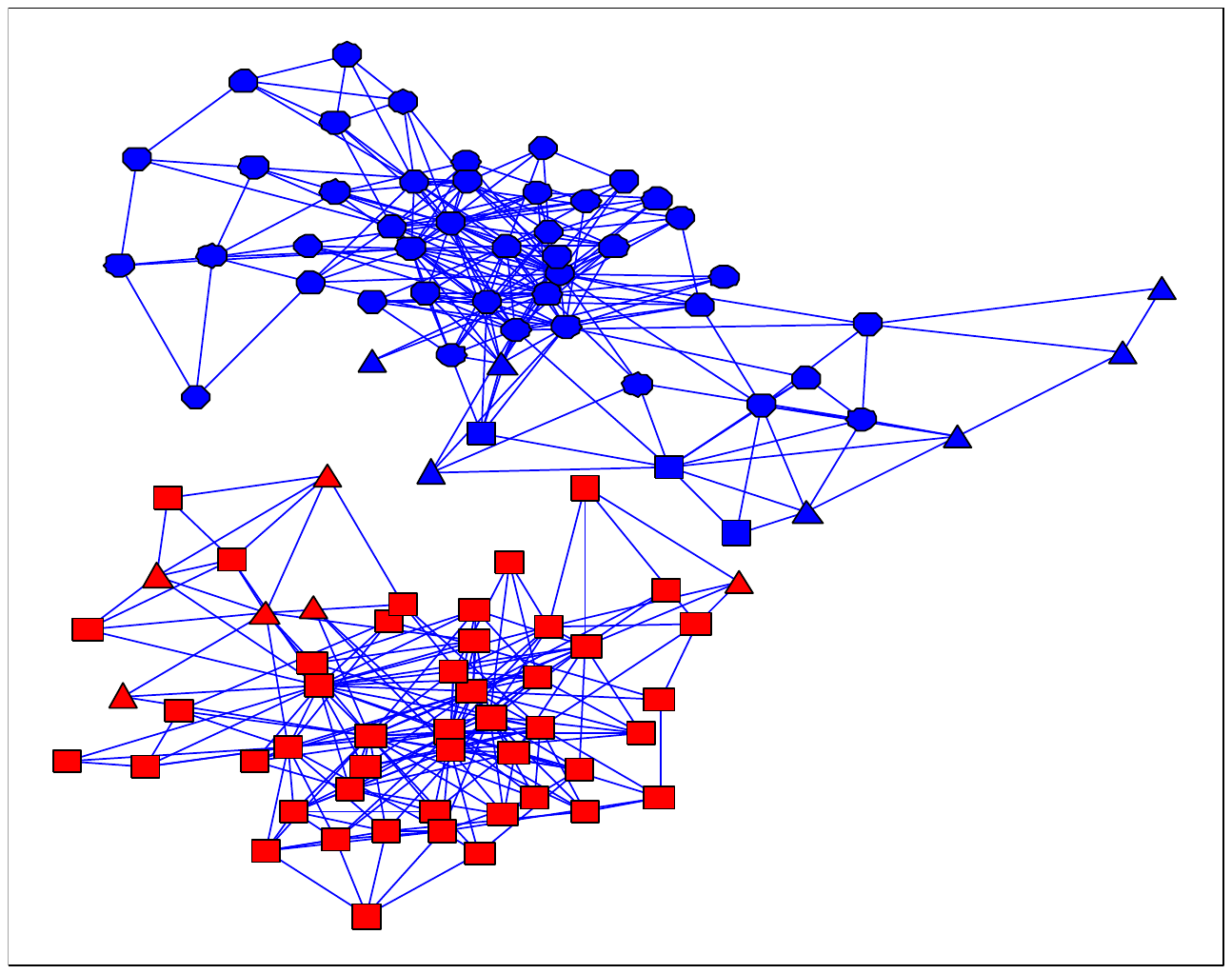}\\
\vspace{-10.2cm}
\\
\hspace{-1.cm}
\includegraphics[angle=0,scale=0.5]{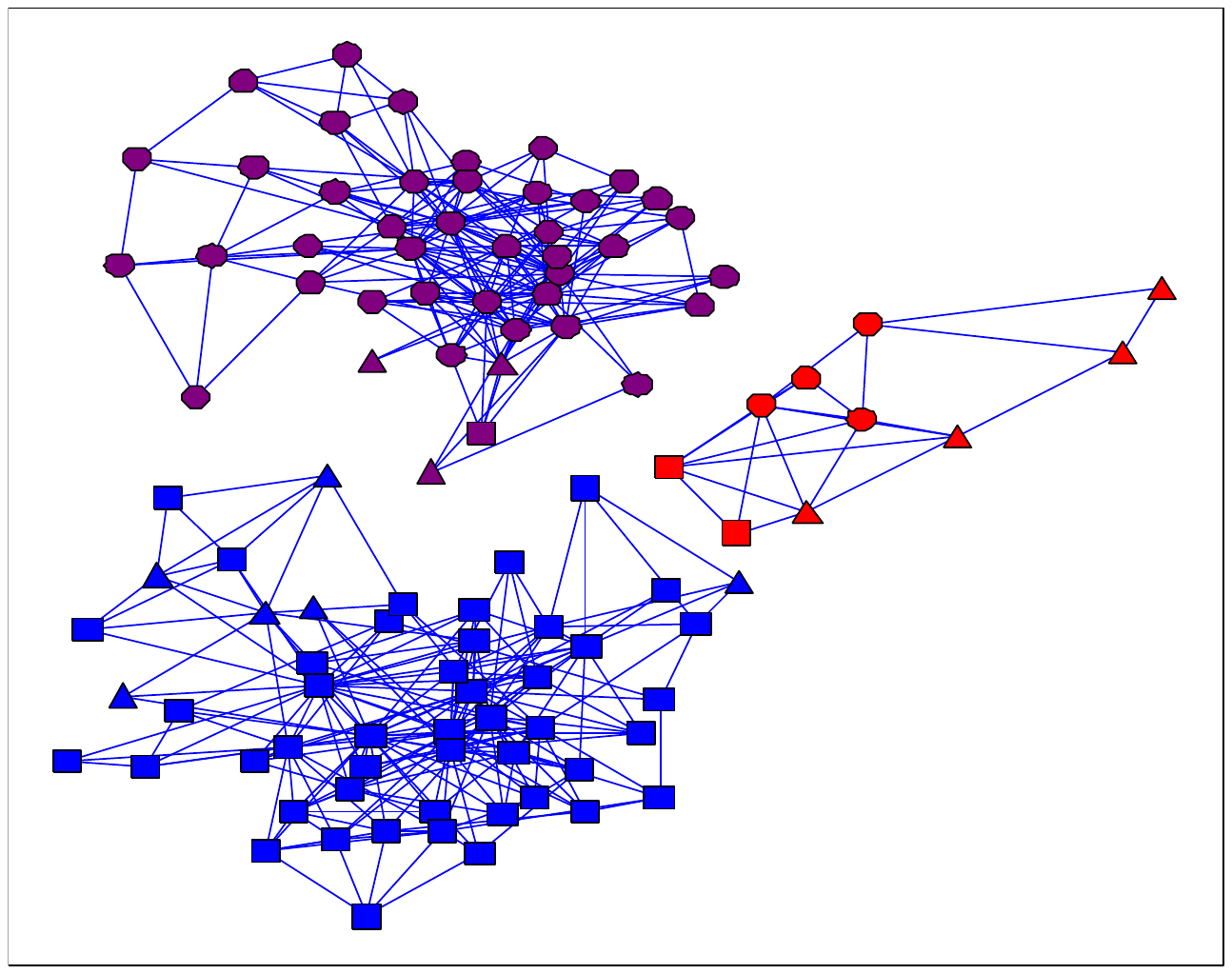}&
\hspace{-4cm}
\includegraphics[angle=0,scale=0.5]{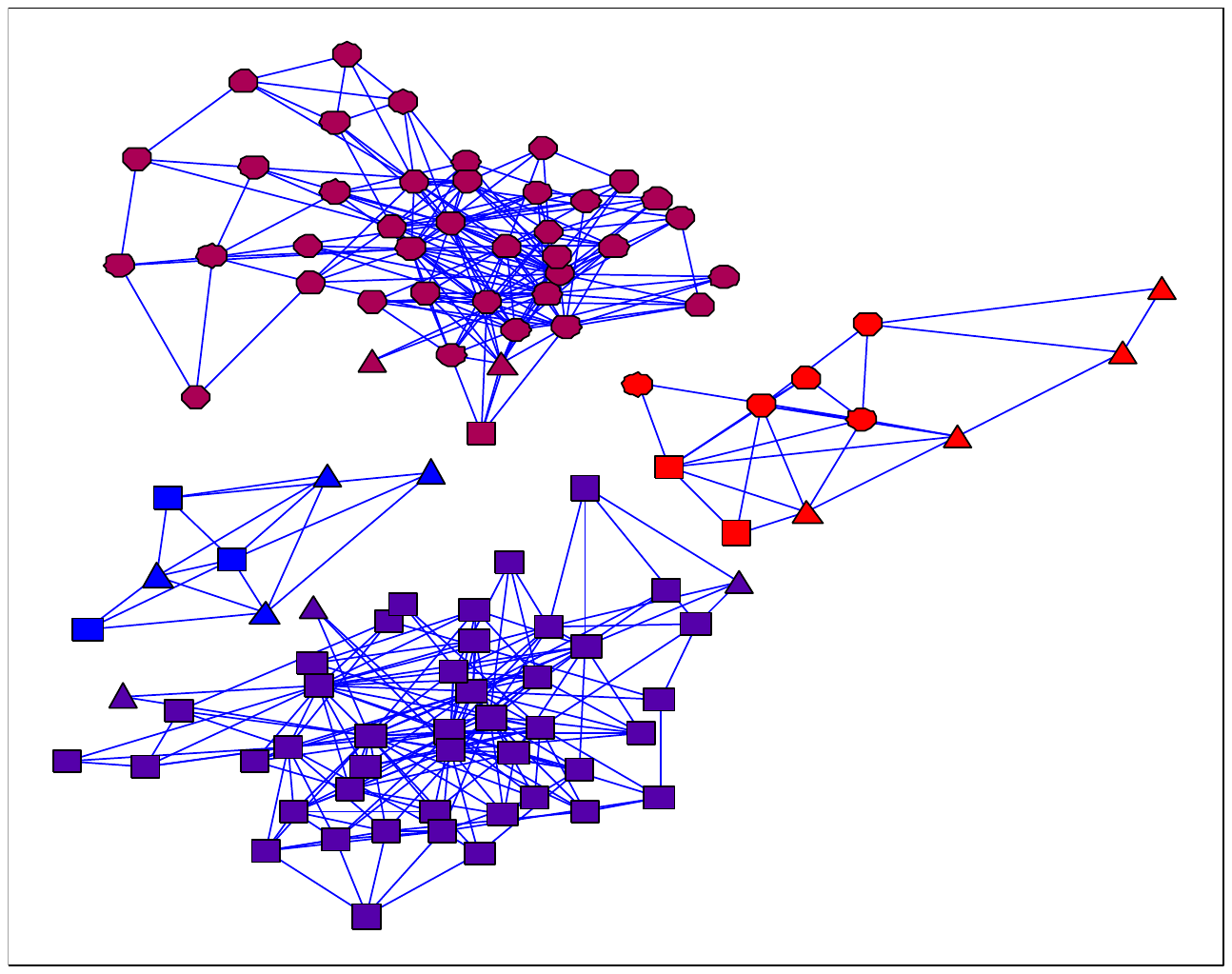}
\end{tabular}
\vspace{-5.2cm} \caption{Graphs $G_{\mathscr{C}}$ for the most
frequently obtained partition of the books network for $R=1$, $\alpha=0.1$: initial graph
(top left), $\delta=0.1$ (top right), $\delta=0.15$ (bottom left),
$\delta=0.2$ (bottom right). Shapes represent political alignment
of the books: circles are liberal, squares are conservative,
triangles are centrist.} \label{fig2}
\end{center}
\vspace{-1cm}
\end{figure}

Let us remark that the computed partitions are solutions to
the Problem~\ref{prob}. Also, for the same value of parameter $\delta$, the
modularity is very similar for all partitions. Actually,
all the partitions obtained for the same value of $\delta$ are
almost the same. As in the previous example, we can see that the choice of parameters $R$ and $\alpha$ affects the probability of obtaining a given partition.
The partition with maximal modularity is obtained for $\delta=0.2$, it is a partition in 4 communities with modularity $0.523$. 
As a comparison, algorithms~\cite{newman2006} and~\cite{blondel2008} obtain partitions in $4$ communities with modularity $0.526$ and $0.527$, respectively.
As we can see, our partition has a modularity that is quite close from those obtained by these algorithms.

In Figure~\ref{fig2}, we represented the graphs
of communities $G_{\mathscr{C}}$ that are the most frequently
obtained for the different values of $\delta$. Let us remark that
even though the information on the political alignment of the
books is not used by the algorithm, our approach allows to uncover
this information. Indeed, for $\delta=0.1$, we obtain $2$
communities that are essentially liberal and conservative. For
$\delta=0.2$, we then obtain $4$ communities: liberal,
conservative, centrist-liberal, centrist-conservative.

In Figure~\ref{fig2a}, we represented the stability of the partitions shown in Figure~\ref{fig2}.
As in the previous example, we can see that the partition with maximal stability changes according to time-scale $t$ which shows that our approach makes it possible to 
detect community at several scales using different values of parameter $\delta$.

\begin{figure}[!h]
\begin{center}
\vspace{-0.5cm}
\includegraphics[angle=0,scale=0.55]{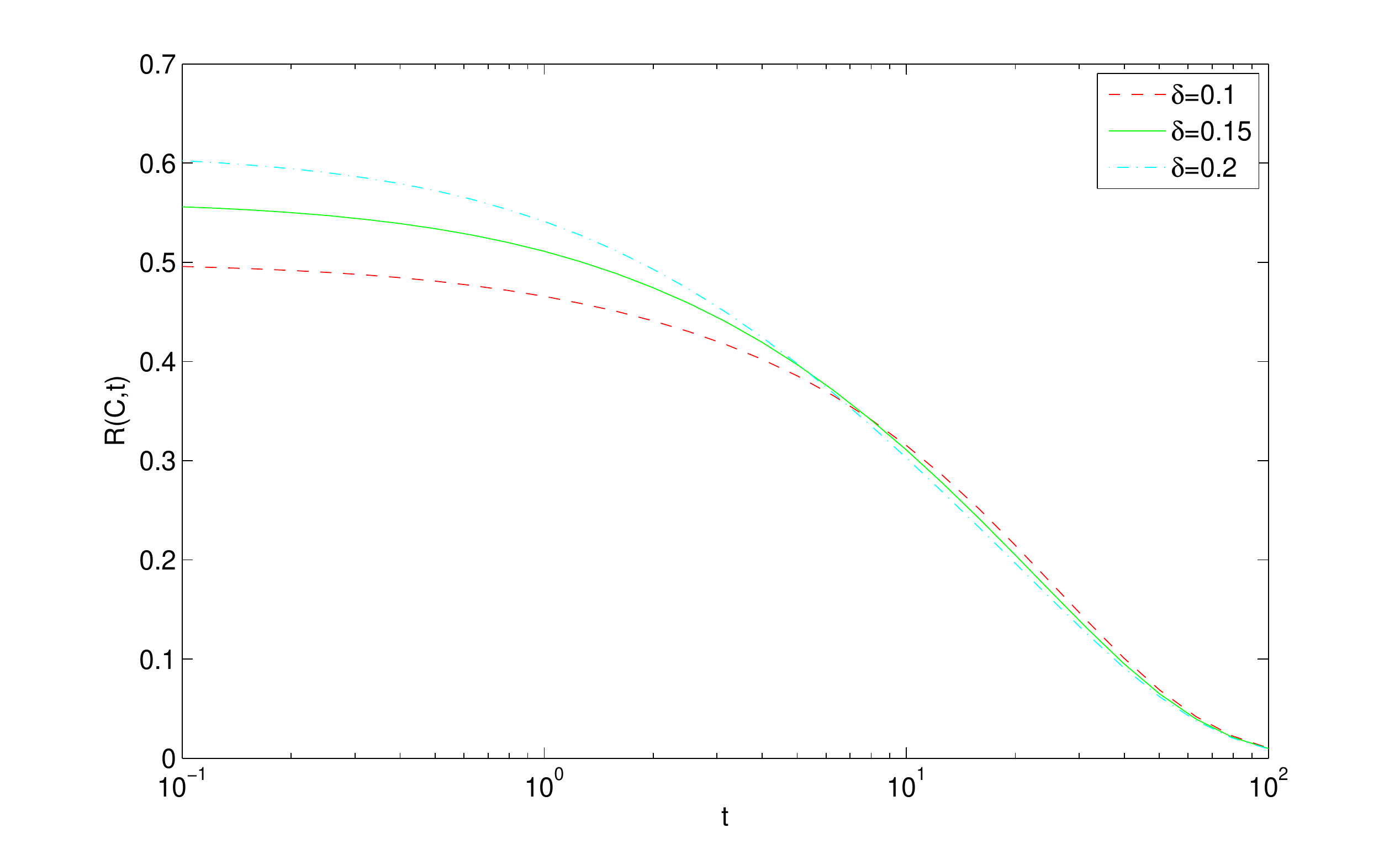}
\vspace{-1cm}
\caption{Stability of the partitions presented in Figure~\ref{fig2}.} \label{fig2a}
\end{center}
\vspace{-1cm}
\end{figure}

\subsubsection{Political blogs} The last example we consider consists of a significantly larger network of $1222$ political blogs~\cite{adamic2005}. 
In this network, an edge between two vertices means that one of the corresponding blogs contained a hyperlink to the other on its front page. 
We also have the information about the political alignment of each blog based on content: 636 are conservative, 586 are liberal.

The two previous examples show that the modularity of the obtained partitions does not depend much on the parameters $R$ and $\alpha$ or on the vector of initial opinions.
For this reason, we decided to apply our opinion dynamics model with parameters $R=1$ and $\alpha=0.1$.  We used $17$ values of $\delta$ between $0.05$ and $0.75$.
The parameter $\rho$ was chosen according to Corollary~\ref{corol}: $\rho=1-\alpha \delta$. For each value of $\delta$, the model was simulated
only once for a vector of initial opinion chosen randomly in $[0,1]^{1222}$. Simulations were performed as long as
enabled by floating point arithmetics.

The partition with maximal modularity was obtained for $\delta=0.4$. It is a partition in $12$ communities with modularity $0.426$.
There are 2 main communities: one with $653$ blogs, from which 94\% are conservative, and one with $541$ blogs, from which 98\% are liberal. 
The $28$ remaining blogs are distributed in $10$ tiny communities. 
When we progressively increase $\delta$, we can see that the size of the two large communities reduces moderately but progressively until $\delta=0.65$
where the conservative community splits into several smaller communities, the largest one containing $40$ blogs.
The liberal community remains until $\delta=0.725$ where it splits into smaller communities, the largest one containing $54$ blogs.

As a comparison,  algorithm~\cite{newman2006} obtains a partition in $2$ communities with modularity $0.426$ whereas 
algorithm~\cite{blondel2008} obtains a partition in $9$ communities with modularity $0.427$.
As we can see, the partition we obtain is very acceptable in terms of modularity. 

In Table~\ref{tabsum}, we give a comparative summary of the modularity of the partition obtained for the three examples by our approach and by the algorithms
presented in~\cite{newman2006,blondel2008}. Though slightly smaller, the modularity of the partition we obtain is comparable to that of other partitions which is actually surprising since our approach, contrarily to~\cite{newman2006,blondel2008} does not try to maximize modularity. 

\begin{table}[!h]
\begin{center}
\begin{tabular}{|c|c|c|c|}
\hline
Network & Karate  & Books  & Blogs  \\
\hline
Number of nodes & 34 & 105 & 1222 \\
\hline
\hline
This article       & 0.417 & 0.523 & 0.426 \\
\cite{newman2006}  & 0.419 & 0.526 & 0.426  \\
\cite{blondel2008} & 0.419 & 0.527 & 0.427 \\
\hline
\end{tabular}
\vspace{0.5cm}
\caption{Modularity of the partitions obtained by the approach presented in this paper and by  the algorithms
presented in~\cite{newman2006,blondel2008} for the three examples considered in this paper.}
\label{tabsum}
\end{center}
\end{table}

%


\section{Conclusion and Future Work}

In this paper, we introduced and analyzed a model of opinion dynamics with decaying confidence where agents may only reach local agreements organizing themselves in communities.
Under suitable assumptions, we have shown that these communities correspond to asymptotically connected components of the network. We have also provided an algebraic characterization of communities in terms of eigenvalues of the matrix defining the collective dynamics.
To complete the analysis of our model, future work should focus on relaxing Assumption~\ref{assum4}
by studying the model behavior when there is an agent $i\in V$ that approaches its limit value at a rate exactly $\rho$:
$$
\limsup_{t\rightarrow +\infty} \frac{1}{t} \log(|x_i(t)-x_i^*|) = \log(\rho).
$$

In the last part of the paper, we have applied our opinion dynamics model to address the problem of community detection in graphs.
We believe that this new approach offers an appealing interpretation of community detection: communities are sets of agents that succeed to reach an agreement under some convergence rate constraint. We have shown on three examples that this approach is not only appealing but is also effective.
In the future, we shall work on a distributed implementation of our approach. Let us remark that
this should be feasible since our approach is by nature based on distributed computations.
Then, we shall use our approach to analyze a number of networks including large scale networks.


\bibliographystyle{alpha}
\bibliography{paper}

\newcommand{\etalchar}[1]{$^{#1}$}
\begin{thebibliography}{BDG{\etalchar{+}}08}

\bibitem[AB08]{angeli2007}
D.~Angeli and P.~A. Bliman.
\newblock Tight estimates for convergence of some non-stationary consensus
  algorithms.
\newblock {\em Systems and Control Letters}, 57(12):996--1004, 2008.

\bibitem[AG05]{adamic2005}
L~Adamic and N.~Glance.
\newblock The political blogosphere and the 2004 u.s. election:divided they
  blog.
\newblock In {\em Conference on Knowledge Discovery in Data: Proceedings of the
  3rd international workshop on Link discovery}, 2005.

\bibitem[BDG{\etalchar{+}}08]{brandes2008}
U.~Brandes, D.~Delling, M.~Gaertler, R.~G\"orke, M.~Hoefer, Z.~Nikoloski, and
  D.~Wagner.
\newblock On modularity clustering.
\newblock {\em IEEE Trans. on Knowledge and Data Engineering}, 20(2):172--188,
  2008.

\bibitem[BGLL08]{blondel2008}
V.D. Blondel, J.-L. Guillaume, R.~Lambiotte, and E.~Lefebvre.
\newblock Fast unfolding of communites in large networks.
\newblock {\em Journal of Statistical Mechanics: Theory and Experiment},
  1742-5468(08):10008+12, 2008.

\bibitem[BHOT05]{blondel2005}
V.~D. Blondel, J.~M. Hendrickx, A.~Olshevsky, and J.N. Tsitsiklis.
\newblock Convergence in multiagent coordination, consensus, and flocking.
\newblock In {\em Proc. IEEE Conf. on Decision and Control}, pages 2996--3000,
  2005.

\bibitem[BHT09]{blondel2007}
V.~D. Blondel, J.~M. Hendrickx, , and J.N. Tsitsiklis.
\newblock On the {2R} conjecture for multi-agent systems.
\newblock {\em IEEE Trans. on Automatic Control}, 54(11):2506--2517, 2009.

\bibitem[CHN86]{cohen1986}
J.E. Cohen, J.~Hajnal, and C.M. Newman.
\newblock Approaching consensus can be delicate when positions harden.
\newblock {\em Stochastic Processes and their Applications}, 22:315--322, 1986.

\bibitem[Chu97]{chung1997}
F.~Chung.
\newblock {\em Spectral Graph Theory}.
\newblock American Mathematical Society, 1997.

\bibitem[CS77]{chatterjee1977}
S.~Chatterjee and E.~Seneta.
\newblock Towards consensus: some convergence theorems on repeated averaging.
\newblock {\em Journal of Applied Probability}, 14(1):89--97, 1977.

\bibitem[For10]{fortunato2009}
S.~Fortunato.
\newblock Community detection in graphs.
\newblock {\em Physics Reports}, 486:75--174, 2010.

\bibitem[HK02]{krause2002}
R.~Hegselmann and U.~Krause.
\newblock Opinion dynamics and bounded confidence models, analysis, and
  simulation.
\newblock {\em Journal of Artificial Societies and Social Simulation}, 5(3),
  2002.

\bibitem[JLM03]{jadbabaie2003}
A.~Jadbabaie, J.~Lin, and A.~S. Morse.
\newblock Coordination of groups of mobile autonomous agents using nearest
  neighbor rules.
\newblock {\em IEEE Trans. on Automatic Control}, 48(6):988--1001, 2003.

\bibitem[Kra97]{krause1997}
U.~Krause.
\newblock Soziale {Dynamiken} mit vielen {Interakteuren}. {Eine}
  {Problemskizze}.
\newblock In {\em Modellierung und Simulation von Dynamiken mit vielen
  interagierenden Akteuren}, pages 37--51, 1997.

\bibitem[LDB09]{lambiotte2009}
R.~Lambiotte, J.-C. Delvenne, and M.~Barahona.
\newblock Laplacian dynamics and multiscale modular structure in networks.
\newblock Technical report, 2009.
\newblock arXiv:0812.1770v3.

\bibitem[Lor05]{lorenz2005}
J.~Lorenz.
\newblock A stabilization theorem for dynamics of continuous opinions.
\newblock {\em Physica A: Statistical Mechanics and its Applications},
  335:217--223, 2005.

\bibitem[Mor05]{moreau2005}
L.~Moreau.
\newblock Stability of multiagent systems with time-dependent communication
  links.
\newblock {\em IEEE Trans. on Automatic Control}, 50(2):169--182, 2005.

\bibitem[New06]{newman2006}
M.~E.~J. Newman.
\newblock Modularity and community structure in networks.
\newblock {\em Proc. Natl. Acad. Sci. USA}, 103:8577--8582, 2006.

\bibitem[NG04]{newman2004}
M.~E.~J. Newman and M.~Girvan.
\newblock Finding and evaluating community structure in networks.
\newblock {\em Phys. Rev. E}, 69:026113, 2004.

\bibitem[OSFM07]{olfati2007}
R.~Olfati-Saber, J.~A. Fax, and R.~M. Murray.
\newblock Consensus and cooperation in networked multi-agent systems.
\newblock {\em Proceedings of the IEEE}, 95(1):215--233, 2007.

\bibitem[OT09]{olshevsky2006}
A.~Oshelvsky and J.~N. Tsitsiklis.
\newblock Convergence speeds in distributed consensus and averaging.
\newblock {\em SIAM J. Control and Optimization}, 48(1):33--55, 2009.

\bibitem[RB05]{ren2005}
W.~Ren and R.~W. Beard.
\newblock Consensus seeking in multiagent systems under dynamically changing
  interaction topologies.
\newblock {\em IEEE Trans. on Automatic Control}, 50(5):655--661, 2005.

\bibitem[Sen81]{seneta1981}
E.~Seneta.
\newblock {\em Non-Negative Matrices and Markov Chains}.
\newblock Springer-Verlag, 1981.

\bibitem[Zac73]{zachary1977}
W.~W. Zachary.
\newblock An information flow model for conflict and fission in small groups.
\newblock {\em Journal of anthropological research}, 33(4):1977, 452-473.

\bibitem[ZW09]{zhou2009}
J.~Zhou and Q.~Wang.
\newblock Convergence speed in distributed consensus over dynamically switching
  random networks.
\newblock {\em Automatica}, 45(6):1455--1461, 2009.

\end{thebibliography}

\end{document}